\documentclass[10pt]{article}
\usepackage{amsmath}
\usepackage{amssymb}
\usepackage{amsthm}
\usepackage{mathrsfs}
\usepackage{color}
\usepackage{eucal}
\usepackage{enumitem}
\usepackage[pdfencoding=auto]{hyperref}
\usepackage[
   backend=biber, 
   sorting=nyt, 
   isbn=false, 
   url=false, 
   doi=false, 
   giveninits=true, 
   maxnames=10 
]{biblatex}
\renewbibmacro{in:}{}
\newtheorem{Def}{Definition}
\newtheorem{Thm}{Theorem}
\newtheorem{Prop}{Proposition}
\newtheorem{Lem}[Prop]{Lemma}
\newtheorem{Cor}[Prop]{Corollary}
\newtheorem{Ques}{Question}
\theoremstyle{remark}
\newtheorem{Rmk}{Remark}
\newtheorem*{Cmt}{Comment}
\renewcommand{\Re}{\operatorname{Re}}
\newcommand{\M}{\CMcal{M}}
\newcommand{\N}{\CMcal{N}}
\newcommand{\tr}{\operatorname{tr}}
\newcommand{\D}[4]{D_{{#3}, {#4}}({#1}||{#2})}
\newcommand{\Q}[4]{Q_{{#3}, {#4}}({#1}||{#2})}
\newcommand{\renyiD}[3]{D_{{#3}}({#1}||{#2})}
\newcommand{\renyiQ}[3]{Q_{{#3}}({#1}||{#2})}
\newcommand{\sandD}[3]{\widetilde{D}_{{#3}}({#1}||{#2})}
\newcommand{\sandQ}[3]{\widetilde{Q}_{{#3}}({#1}||{#2})}
\title{On $\alpha$-$z$-R\'{e}nyi divergence
in the von Neumann algebra setting}
\author{Shinya KATO\thanks{E-mail:kato.shinya.k6@s.mail.nagoya-u.ac.jp}}
\date{Graduate School of Mathematics, Nagoya University, Furocho, Chikusaku, Nagoya, 464-8602,
Japan}
\begin{document}
\maketitle
\begin{abstract}
We will investigate the $\alpha$-$z$-R\'{e}nyi divergence 
in the general von Neumann algebra setting 
based on Haagerup non-commutative $L^p$-spaces.
In particular, we establish almost all its expected properties 
when $0 < \alpha < 1$ 
and some of them when $\alpha > 1$. 
In an appendix we also give an equality condition for 
generalized H\"{o}lder's inequality 
in Haagerup non-commutative $L^p$-spaces.
\end{abstract}
\section{Introduction}
The QIT (quantum information theory)  can mathematically be described 
in terms of operator algebras on finite-dimensional Hilbert spaces, 
and quantum divergences, such as the relative entropy, 
are functionals on the pairs of states 
(or density  matrices) $\rho$ and $\sigma$ on a fixed operator algebra.
\par
The quantum version of the Kullback--Leibler divergence
was introduced by Umegaki \cite{Umegaki62} 
as the \emph{relative entropy} $D(\rho || \sigma)$.
We also have
the quantum version of $\alpha$-R\'{e}nyi divergence with 
$\alpha \in (0, \infty)\setminus \{1\}$. 
The latter divergence, 
denoted by $\renyiD{\rho}{\sigma}{\alpha}$, was first introduced 
by Petz \cite{Petz86} as one of the \emph{quasi-entropies}. 
This divergence has most of the properties the 
classical R\'{e}nyi entropy does.
However, $\renyiD{\rho}{\sigma}{\alpha}$ satisfies 
the \emph{DPI} (\emph{Data Processing Inequality}) 
$\renyiD{\Phi(\rho)}{\Phi(\sigma)}{\alpha} 
\le \renyiD{\rho}{\sigma}{\alpha}$ with \emph{CPTP} 
(\emph{completely positive trace-preserving}) map 
(or \emph{quantum channel}) $\Phi$ only
if $\alpha$ falls into $(0, 2] \setminus\{1\}$.
Another quantum $\alpha$-R\'{e}nyi divergence 
$\sandD{\rho}{\sigma}{\alpha}$
was introduced by 
M\"{u}ller-Lennert, Dupuis, Szehr, Fehr and Tomamichel \cite{MLDS13} 
and Wilde, Winter and Yang \cite{WWY14} independetly, and 
called the \emph{sandwiched R\'{e}nyi divergence}.
The divergence $\sandD{\rho}{\sigma}{\alpha}$ satisfies the DPI
for all $\alpha \in [1/2, \infty)\setminus\{1\}$
due to Beigi \cite{Beigi13} (who treated only the case of $\alpha > 1$)
and Frank and Lieb \cite{FrankLieb13} independently.
\par
The \emph{$\alpha$-$z$-R\'{e}nyi divergence} $\D{\rho}{\sigma}{\alpha}{z}$ 
was introduced by Audenaert and Datta \cite{AudenaertDatta15}
as a simultaneous generalization of 
the R\'{e}nyi divergence and the sandwiched R\'{e}nyi divergence. 
(We remark that the same quantity had been introduced 
by Jak\v{s}i\'{c}, Ogata, Pautrat and Pillet \cite[Section 4.4.3]{jaksicetal11} 
as a certain entropic functional before the work \cite{AudenaertDatta15}.)
Audenaert and Datta posed the problem of determinating  
all the $(\alpha, z)$ such that the $\alpha$-$z$-R\'{e}nyi divergence 
satisfies the DPI. 
The problem was investigated by 
Carlen, Frank and Lieb \cite{CarlenFrankLieb18} 
and finally settled completely 
by Zhang \cite{Zhang20} using a variational expression.
\par
Beyond the finite dimensional setting, 
many quantum divergences have been generalized even 
to the von Neumann algebra setup.
The relative entropy in the general von Neumann algebra setup 
introduced by Araki \cite{Araki76}, \cite{Araki77}
(using relative modular operators) is such a typical example.
The R\'{e}nyi divergence 
was also generalized to the von Neumann algebra setting in \cite{Hiai18}
as a variant of the \emph{standard $f$-divergence}, 
which is a special case of 
Petz's quasi-entorpies \cite{Petz85}, \cite{Petz86}. 
Then, Berta, Scholz and Tomamichel \cite{BST18} 
and Jen\v{c}ov\'{a} \cite{Jencova18}, \cite{Jencova21}
introduced the sandwiched R\'{e}nyi divergence in the von Neumann algebra setup.
Berta et al.~used 
Araki--Masuda non-commutative $L^p$-spaces 
to introduce the sandwiched R\'{e}nyi divergence
(which they called the Araki--Masuda divergence in \cite{BST18}) 
and Jen\v{c}ov\'{a} used Kosaki non-commutative $L^p$-spaces. 
It was established in \cite{Jencova18}, \cite{Jencova21}
(see also \cite[Theorem 3.11]{Hiai21}) that both the approaches 
to the sandwiched R\'{e}nyi divergence indeed define the same quantity. 
The reader can also find an explanation of those approaches 
in Hiai's monograph \cite[Remark 3.14]{Hiai21} 
(by utilizing Haagerup non-commutative $L^p$-spaces).
Recently, the $\alpha$-$z$-R\'{e}nyi divergence was also generalized to 
the infinite-dimensional \emph{type I} von Neumann algebra setup 
by Mosonyi \cite{Mosonyi23} when $\alpha > 1$ 
and by Zhang and Qi \cite{ZhangQi23} when $0<\alpha<1$.
\par
In our previous paper \cite{KatoUeda23},
we actually proposed a possible definition of 
the $\alpha$-$z$-R\'{e}nyi divergence in the general von Neumann algebra setup  
by using Haagerup non-commutative $L^p$-spaces. 
However, we discussed there only a few facts on 
the $\alpha$-$z$-R\'{e}nyi divergence 
because our purpose there was to illustrate 
how useful our result on non-commutative $L^p$-spaces is. 
Thus, our definition has not yet been justified. 
Hence, we will examine fundamental properties of 
the $\alpha$-$z$-R\'{e}nyi divergence under our definition. 
Actually, we will establish almost all the expected properties 
when $0 < \alpha <1$ (Theorem \ref{Thm1}), 
and will do some of those properties when $\alpha > 1$ (Theorem \ref{Thm2}). 
We also try to clarify next tasks concerning the case of $\alpha > 1$ 
by posing several questions. 
In \ref{Appendix Holder} we will give an equality condition 
for H\"{o}lder's inequality in the framework of Haagerup $L^p$-spaces.
This seems to be a new result.
\section{Preliminaries}
\subsection{Notations}
In this paper, let $\M$ be a von Neumann algebra 
and $1_{\M}$ be the unit of $\M$. 
When no confusion is possible, we will simply denote by $1$ the unit.
We denote by $\M_*$ the predual of $\M$ and
by $\M_*^+$ its positive cone. 
For $\varphi \in \M_*^+$, 
we denote the support projection of $\varphi$ by $s(\varphi)$. 
For $0< p \le \infty$, 
we denote the Haagerup non-commutative $L^p$-space 
associated with $\M$ by $L^p(\M)$
and its positive cone by $L^p(\M)_+$.
There is a bijective linear isomorphism, called 
the \emph{Haagerup correspondence}, 
between $\M_*$ and $L^1(\M)$ by $\varphi \longmapsto h_{\varphi}$,
and the trace-like functional $\tr \colon L^1(\M) \longrightarrow \mathbb{C}$ 
is defined by $\tr(h_{\varphi}) = \varphi(1)$. 
The details on these materials can be found in \cite{Terp81}, \cite{HiaiLectures21}. 
For $p \in [1, \infty]$, 
we denote the symmetric Kosaki non-commutative $L^p$-space associated with $\M$ 
with respect to a faithful normal state $\varphi_0$
by $L^p(\M, \varphi_0)_{1/2}$ and its norm 
by $\|\cdot \|_{p, \varphi_0, 1/2}$.
Let $p$, $q \in [1, \infty]$ be given with $1/p + 1/q = 1$.
Then, the Haagerup and the Kosaki non-commutative $L^p$-spaces are
isometrically isomorphic to each other by
$L^p(\M) \ni a \longmapsto h_{\varphi_0}^{1/2q} a h_{\varphi_0}^{1/2q}
\in L^p(\M, \varphi_0)_{1/2}$ ($\subset L^1(M)$).
\subsection{Some Lemmas}
We will provide some lemmas that will be necessary later.
\par
The next lemma is (generalized) H\"{o}lder's inequality; 
see \cite[Proposition 9.17]{HiaiLectures21} for its proof.
\begin{Lem}
    Let $p$, $q$, $r \in (0, \infty]$ with
    $1/r = 1/p + 1/q$.
    If $a \in L^p(\M)$ and $b \in L^q(\M)$,
    then $ab \in L^r(\M)$ and
    \begin{equation*}
        \|ab\|_r \le \|a\|_p \|b\|_q.
    \end{equation*}
\end{Lem}
The next lemma was essentially given in e.g., \cite{FackKosaki86}.
\begin{Lem} \label{lp norm order}
    Let $0 < p \le \infty$ and $a$, $b \in L^p(\M)_+$. 
    If $a \le b$, 
    then $\|a \|_p \le \|b\|_p$.
\end{Lem}
\begin{proof}
    Assume that $a, b \in L^p(\M)_+$ with $a \le b$.
    By \cite[Lemma 2.5(iii)]{FackKosaki86}, 
    we have $\mu_t(a) \le \mu_t(b)$ for any $t > 0$, 
    where $\mu_t$ is the ($t$th) generalized $s$-number 
    (cf.~\cite[Definition 2.1]{FackKosaki86}).
    In addition, $\mu_t(a)$ and $\|a\|_p = \tr(|a|^p)^{1/p}$ have the relation
    that $\mu_t(a) = t^{-1/p}\|a\|_p$ for any $t > 0$
    (cf.~\cite[Lemma 4.8]{FackKosaki86}). 
    Therefore, $\|a\|_p \le \|b\|_p$ holds.
\end{proof}
The next lemma is often presented without proof. 
\begin{Lem} \label{trace equality}
    Let $\alpha > 0$ and $p$, $q \ge 1$ be such that $1/p + 1/ q = 1$.
    Let $a \in L^{\alpha p} (\M)_+$ and $b \in L^{\alpha q}(\M)_+$ 
    (hence, $a^{1/2}ba^{1/2}$, $b^{1/2}ab^{1/2} \in L^{\alpha}(\M)_+$). 
    Then, we have $\tr((a^{1/2}ba^{1/2})^\alpha) 
    = \tr((b^{1/2}ab^{1/2})^\alpha)$.
\end{Lem}
\begin{proof}
    We note that
    $\tr((a^{1/2}ba^{1/2})^\alpha) = \|b^{1/2}a^{1/2}\|_{2\alpha}^{2\alpha}$
    and $\tr((b^{1/2}ab^{1/2})^\alpha) 
    = \|a^{1/2}b^{1/2}\|_{2\alpha}^{2\alpha}$.
    By \cite[Lemma 2.5(ii)]{FackKosaki86} and \cite[Lemma 4.8]{FackKosaki86}, 
    we have
    \begin{equation*}
        \|a^{1/2}b^{1/2}\|_{2\alpha}^{2\alpha}
        = t^{1/p}\mu_t (a^{1/2}b^{1/2}) = 
        t^{1/p}\mu_t (b^{1/2}a^{1/2}) = 
        \|b^{1/2}a^{1/2}\|_{2\alpha}^{2\alpha}. 
    \end{equation*}
    Thus, the desired equality has been shown.
\end{proof}
The next lemma is due to Fack--Kosaki \cite[Theorem 4.9(iii)]{FackKosaki86}.
\begin{Lem} \label{FackKosaki thm 4.9}
    Let $p \in (0, 1]$ and $a$, $b \in L^p(\M)$.
    Then, we have
    \begin{equation*} 
        \| a + b \|_p^p \le \|a \|_p^p + \|b\|_p^p.
    \end{equation*}
\end{Lem}
The next lemma is Kosaki's generalized Powers--St\o mer inequality 
\cite[Appendix]{HiaiNakamura89}.
\begin{Lem} \label{powers stomer}
    If $0 <\theta \le 1$, $\theta \le p \le \infty$ and $a,b \in L^p(\M)_+$,
    then
    \begin{equation*} 
        \| a^\theta - b^\theta \|_{p/\theta} 
        \le \| a - b \|_p^\theta.
    \end{equation*}
\end{Lem}
The continuity of the map $L^1(\M)_+ \ni h \longmapsto h^{1/p} \in L^p(\M)_+$ 
is given when $p \ge 1$ in \cite[Theorem 4.2]{Kosaki84}.
\begin{Lem} \label{norm continuity}
    Let $p \in (0, \infty)$ and $\varphi$, $\varphi_n \in \M_*^+$.
    If $\lim_{n \to \infty}\| h_{\varphi_n} - h_{\varphi} \|_1 = 0$, 
    then $\lim_{n \to \infty} 
    \| h_{\varphi_n}^{1/p} - h_{\varphi}^{1/p} \|_p = 0$.
    Namely, $L^1(\M)_+ \ni h \longmapsto h^{1/p} \in L^p(\M)_+$ 
    is a continuous map in the norm topology.
\end{Lem}
\begin{proof}
    When $p \in [1, \infty)$, this is clear 
    by Lemma \ref{powers stomer}
    with $\theta = 1/p$, $p=1$. 
    \par
    We then consider the case of $0 < p < 1$. 
    We can choose a natural number $k \in \mathbb{N}$ with
    $1/kp \le 1$.
    By Lemma \ref{FackKosaki thm 4.9},
    we have
    \begin{align*}
        \|h_{\varphi_n}^{1/p} - h_{\varphi}^{1/p}\|_p^p
        &= \|(h_{\varphi_n}^{1/kp} - h_{\varphi}^{1/kp})
        h_{\varphi_n}^{(k-1)/kp}
        + h_{\varphi}^{1/kp} 
        (h_{\varphi_n}^{1/kp} - h_{\varphi}^{1/kp})
        h_{\varphi_n}^{(k-2)/kp} 
        + \cdots \\
        &\qquad\qquad+ h_{\varphi}^{(k-2)/kp}
        (h_{\varphi_n}^{1/kp} - h_{\varphi}^{1/kp})
        h_{\varphi_n}^{1/kp}
        + h_{\varphi}^{(k-1)/kp}
        (h_{\varphi_n}^{1/kp} - h_{\varphi}^{1/kp}) \|_p^p \\
        &\le \| (h_{\varphi_n}^{1/kp} - h_{\varphi}^{1/kp})
        h_{\varphi_n}^{(k-1)/kp} \|_p^p
        + \| h_{\varphi}^{1/kp} 
        (h_{\varphi_n}^{1/kp} - h_{\varphi}^{1/kp})
        h_{\varphi_n}^{(k-2)/kp} \|_p^p
        + \cdots \\
        &\qquad\qquad+ \| h_{\varphi}^{(k-2)/kp}
        (h_{\varphi_n}^{1/kp} - h_{\varphi}^{1/kp})
        h_{\varphi_n}^{1/kp} \|_p^p
        + \| h_{\varphi}^{(k-1)/kp}
        (h_{\varphi_n}^{1/kp} - h_{\varphi}^{1/kp}) \|_p^p. 
    \end{align*}
    (n.b., $h^\alpha h^\beta = h^{\alpha + \beta}$ naturally holds
    for any $\alpha$, $\beta \ge 0$; see \cite[Proposition B.2]{Hiai21}.)
    Applying
    H\"{o}lder's inequality and Lemma \ref{powers stomer}
    to each term of the final line of the above inequality,
    we obtain that
    \begin{align*}
        &\| (h_{\varphi_n}^{1/kp} - h_{\varphi}^{1/kp}) 
        h_{\varphi_n}^{(k-1)/kp} \|_p^p
        \le \| h_{\varphi_n} - h_{\varphi}\|_1^{1/k}
        \| h_{\varphi_n}^{(k-1)/kp} \|_{kp/(k-1)}^p, \\
        &\| h_{\varphi}^{1/kp} 
        (h_{\varphi_n}^{1/kp} - h_{\varphi}^{1/kp})
        h_{\varphi_n}^{(k-2)/kp} \|_p^p
        \le  \| h_{\varphi}^{1/kp} \|_{kp}^p
        \| h_{\varphi_n} - h_{\varphi} \|_1^{1/k}
        \| h_{\varphi_n}^{(k-2)/kp} \|_{kp/(k-2)}^p, \\
        &\qquad\vdots \\
        &\| h_{\varphi}^{(k-2)/kp}
        (h_{\varphi_n}^{1/kp} - h_{\varphi}^{1/kp})
        h_{\varphi_n}^{1/kp} \|_p^p 
        \le \| h_{\varphi}^{(k-2)/kp} \|_{kp/(k-2)}^p
        \| h_{\varphi_n} - h_{\varphi} \|_1^{1/k}
        \| h_{\varphi_n}^{1/kp} \|_{kp}^p, \\
        &\| h_{\varphi}^{(k-1)/kp}
        (h_{\varphi_n}^{1/kp} - h_{\varphi}^{1/kp}) \|_p^p
        \le \| h_{\varphi}^{(k-1)/kp} \|_{kp/(k-1)}^p
        \| h_{\varphi_n} - h_{\varphi} \|_1^{1/k}.
    \end{align*}
    Therefore, 
    if $\| h_{\varphi_n} - h_{\varphi} \|_1 \to 0$,
    then every term converges to $0$.
    Thus, $\|h_{\varphi_n}^{1/p} - h_{\varphi}^{1/p} \|_p^p \to 0$
    as $n \to \infty$. 
\end{proof}
\section{Definition of the \texorpdfstring{$\alpha$-$z$-R\'{e}nyi}{alpha-z-R\'{e}nyi} divergence and Relation to other divergences}
The $\alpha$-$z$-R\'{e}nyi divergence for normal positive linear functionals 
on (general) von Neumann algebras is defined in \cite[\S 4]{KatoUeda23}.
\begin{Def}
    For $\psi$, $\varphi \in \M^+_*$ 
    and $\alpha$, $z > 0$ with $\alpha \ne 1$, let 
    \begin{align*}
        \Q{\psi}{\varphi}{\alpha}{z} =
        \begin{cases}
            \tr \left( (h_{\varphi}^{(1-\alpha)/2z} h_{\psi}^{\alpha/z} h_{\varphi}^{(1-\alpha)/2z})^z \right) &(0<\alpha < 1), \\
            \| x \|_z^z &(\text{$\alpha > 1$ \rm{and \eqref{alpha z identity} holds with} $x \in s(\varphi) L^z(\M) s(\varphi)$}), \\
            \infty &(\mathrm{otherwise}),
        \end{cases}
    \end{align*}
    where
    \begin{equation} \label{alpha z identity}\tag{$\spadesuit$}
        h_{\psi}^{\alpha/z} = h_{\varphi}^{(\alpha - 1)/2z} x h_{\varphi}^{(\alpha - 1)/2z}.
    \end{equation}
    When $\psi \ne 0$, 
    the \emph{$\alpha$-$z$-R\'{e}nyi divergence} 
    $\D{\psi}{\varphi}{\alpha}{z}$ is defined by
    \begin{equation*}
        \D{\psi}{\varphi}{\alpha}{z} 
        = \frac{1}{\alpha - 1} \log \frac{\Q{\psi}{\varphi}{\alpha}{z}}{\psi(1)}.
    \end{equation*}
\end{Def}
\begin{Rmk}
    The $\alpha$-$z$-R\'{e}nyi divergence is well defined as discussed in \cite[\S 4]{KatoUeda23}. 
\end{Rmk}
\begin{Rmk} \label{basical remark of alpha z renyi}
    \begin{enumerate}[label=(\roman*)]
        \item \label{basical remark of alpha z renyi 1}
        We include $\psi = 0$ in the definition of 
        $\Q{\psi}{\varphi}{\alpha}{z}$ 
        (only defined for $\psi \ne 0$ in \cite[\S 4]{KatoUeda23})
        to give properties of 
        $\Q{\psi}{\varphi}{\alpha}{z}$ later.
        It is immediately seen that $\Q{0}{\varphi}{\alpha}{z} = 0$
        for any $\varphi$, $\alpha$, $z$. 
        \item \label{basical remark of alpha z renyi 2}
        We use conventions that $\log 0 = -\infty$, $\log \infty = \infty$ 
        and $0\cdot \infty = 0$.
        \item The non-normalized $\alpha$-$z$-R\'{e}nyi divergence
        \begin{equation*}
            \widehat{D}_{\alpha, z}(\psi||\varphi)
            := \frac{1}{\alpha - 1}\log \Q{\psi}{\varphi}{\alpha}{z}
        \end{equation*}
        is often used instead of $\D{\psi}{\varphi}{\alpha}{z}$. 
        One of the reasons to use the non-normalized R\'{e}nyi divergence is 
        that it is more natural in the study of the strong converse exponent,
        see \cite[Remark 2.1]{HiaiMosonyi23}.
        We need to change
        the properties \ref{scaling property1}, \ref{order relation1}, 
        \ref{jointly convexity1} in Theorems \ref{Thm1} and \ref{Thm2}
        sightly for the non-normalized $\alpha$-$z$-R\'{e}nyi diergence. 
        \item Identity \eqref{alpha z identity} and 
        $x \in s(\varphi) L^z(\M)s(\varphi)$ imply $x \ge 0$, that is, 
        $x$ must be positive.
    \end{enumerate}
\end{Rmk}
\begin{Lem} \label{identity lemma}
    Let $\psi$, $\varphi \in \M_*^+$ and $\alpha > 1$.
    There exists $x \in s(\varphi) L^z(\M)s(\varphi)$
    such that identity \eqref{alpha z identity} holds
    if and only if there exists $y \in L^{2z}(\M)s(\varphi)$
    such that the following identity holds:
    \begin{equation} \label{alpha z identity prime} \tag{$\spadesuit$'}
        h_{\psi}^{\alpha/2z} = y h_{\varphi}^{(\alpha - 1)/2z}.
    \end{equation}
    Moreover, $\|x\|_z^z = \|y\|_{2z}^{2z}$ holds in this case.
\end{Lem}
\begin{proof}
    Assume that there exists $y \in L^{2z}(\M)s(\varphi)$ such that 
    identity \eqref{alpha z identity prime} holds. 
    Taking the adjoint of \eqref{alpha z identity prime},
    we have $h_{\psi}^{\alpha/2z} = h_{\varphi}^{(\alpha - 1)/2z} y^*$.
    Multiplying both sides of this equality by 
    \eqref{alpha z identity prime}, 
    we obtain $h_{\psi}^{\alpha} 
    = h_{\varphi}^{(\alpha - 1)/2} y^* y h_{\varphi}^{(\alpha - 1)/2}$
    and  $x = y^* y \in s(\varphi)L^z(\M) s(\varphi)$ 
    since $y \in L^{2z}(\M)s(\varphi)$.
    Therefore, the \emph{if} part is shown.
    \par 
    We next consider the \emph{only if} part.
    Assume that there exists $x \in s(\varphi)L^z(\M)s(\varphi)$ 
    such that identity \eqref{alpha z identity} holds.
    Here, remark that we may and do assume $x \in (s(\varphi)L^z(\M)s(\varphi))_+$. 
    Taking the polar decompositon of $x^{1/2}h_{\varphi}^{(\alpha - 1)/2z}$, 
    we have
    \begin{equation} \label{polar decomp of the identity}
        x^{1/2}h_{\varphi}^{(\alpha - 1)/2z} 
        = u |x^{1/2}h_{\varphi}^{(\alpha - 1)/2z}|,
    \end{equation}
    where $u \in \M$ is a partial isometry with 
    $u^* u = s(|x^{1/2}h_{\varphi}^{(\alpha - 1)/2z}|)$. 
    Then, we obtain
    $h_{\psi}^{\alpha/2z} = |x^{1/2}h_{\varphi}^{(\alpha - 1)/2z}|
    = u^* x^{1/2}h_{\varphi}^{(\alpha - 1)/2z}$.
    As a result, we can take $y = u^* x^{1/2} \in L^{2z}(\M)s(\varphi)$ 
    such that 
    identity \eqref{alpha z identity prime} holds.
    Therefore, the claim is shown.
    \par
    In this case, the last equality holds by 
    $\|y\|_{2z}^{2z} = \tr((y^* y)^{z})
    = \tr((x^{1/2}uu^*x^{1/2})^z) = \tr(x^z) 
    = \|x \|_z^z$, 
    where $u$ is the partial isometry in \eqref{polar decomp of the identity}.
\end{proof}
\begin{Rmk}
    As the same argument of \cite[Lemma 8]{KatoUeda23}, 
    identity \eqref{alpha z identity prime} with $\alpha > 1$
    uniquely determines
    $y \in L^{2z}(\M) s(\varphi)$ if it exists.
\end{Rmk}
\begin{Rmk} \label{identity lead support proj}
    The condition that there exists $x \in s(\varphi) L^z(\M)s(\varphi)$
    such that identity \eqref{alpha z identity} holds
    (equivalently, there exists $y \in L^{2z}(\M)s(\varphi)$
    such that the identity \eqref{alpha z identity prime} holds)
    leads to $s(\psi) \le s(\varphi)$.
    We can see it by multiplying $s(\varphi)$ to \eqref{alpha z identity} 
    from right and left 
    (resp.~multiplying $s(\varphi)$ 
    to \eqref{alpha z identity prime} from right and taking the adjoint).
\end{Rmk}
The $\alpha$-$z$-R\'{e}nyi divergence 
is a two parameter simultaneous generalization of 
(the Petz type or standard) R\'{e}nyi divergence 
(cf.~\cite{Petz85}, \cite{Hiai18})
and the sandwiched R\'{e}nyi divergence 
(cf.~\cite{BST18}, \cite{Jencova18}, \cite{Jencova21}).
(Both the divergences are summarized in \cite[Chap.~3]{Hiai21}.)
This fact is referred to \cite[Sec.~II]{AudenaertDatta15} for the finite dimensional case 
and to \cite[pp.~92--93]{Mosonyi23} for the infinite dimensional type I case.
Similarly, our $\alpha$-$z$-R\'{e}nyi divergence is also 
a simultaneous generalization of these divergences
(though partially stated in \cite[Lemma 9]{KatoUeda23}). 
\par
At first, we recall the concepts of R\'{e}nyi divergence and sandwiched R\'{e}nyi divergence. 
Here we employ the explicit description of 
$\renyiQ{\psi}{\varphi}{\alpha}$ in \cite[Theorem 3.6]{Hiai21}.
For $\psi$, $\varphi \in \M^+_*$ and $\alpha > 0$ with $\alpha \ne 1$, 
we put 
\begin{equation*}
    \renyiQ{\psi}{\varphi}{\alpha} = 
    \begin{cases}
        \tr \left( h_{\psi}^{\alpha} h_{\varphi}^{1-\alpha} \right) &(0<\alpha < 1), \\
        \| \eta \|_2^2 &(\text{$\alpha > 1$ 
        and \eqref{renyi identity} holds with $\eta \in L^2(\M)s(\varphi)$}), \\
        \infty &(\text{otherwise}),
    \end{cases}
\end{equation*}
where
\begin{equation}\label{renyi identity} \tag{$\clubsuit$}
    h_{\psi}^{\alpha/2} = \eta h_{\varphi}^{(\alpha -1)/2},
\end{equation}
and  
\begin{equation*}
    \sandQ{\psi}{\varphi}{\alpha} = 
    \begin{cases}
        \tr \left( (h_{\varphi}^{(1-\alpha)/2\alpha} h_{\psi} h_{\varphi}^{(1-\alpha)/2\alpha})^\alpha \right) &(0 < \alpha < 1), \\
        \| h_{\psi} \|_{\alpha, \varphi, 1/2}^{\alpha} &(\text{$\alpha > 1$, $s(\psi) \le s(\varphi)$ and $h_{\psi} \in L^{\alpha}(\M, \varphi)_{1/2}$}), \\
        \infty &(\text{otherwise}).
    \end{cases}
\end{equation*}
When $\psi \ne 0$, 
the \emph{R\'{e}nyi divergence $\renyiD{\psi}{\varphi}{\alpha}$} and 
the \emph{sandwiched R\'{e}nyi divergence $\sandD{\psi}{\varphi}{\alpha}$} are defined by
\begin{align*}
    \renyiD{\psi}{\varphi}{\alpha} 
    = \frac{1}{\alpha - 1} \log \frac{\renyiQ{\psi}{\varphi}{\alpha}}{\psi(1)}, 
    \qquad \sandD{\psi}{\varphi}{\alpha} 
    = \frac{1}{\alpha - 1} \log \frac{\sandQ{\psi}{\varphi}{\alpha}}{\psi(1)}.
\end{align*}
\begin{Rmk} 
    \begin{enumerate}[label= (\roman*)]
        \item The condition of support projection is assumed 
        in the definition of the R\'{e}nyi divergence for $\alpha > 1$
        in \cite[Definition 3.1]{Hiai21}. 
        However, identity \eqref{renyi identity} leads to
        $s(\psi) \le s(\varphi)$ as Remark \ref{identity lead support proj}.
        (See also the next lemma.)
        Hence, we do not need to put the condition $s(\psi)\le s(\varphi)$ 
        in the definition of $\renyiQ{\psi}{\varphi}{\alpha}$ for $\alpha >1$.
        \item The sandwiched R\'{e}nyi divergence is only defined for 
        $\alpha \in [1/2, \infty) \setminus \{1\}$ in \cite{BST18}, \cite{Jencova21} 
        since DPI (Data Processing Inequality) does not hold for 
        $0 < \alpha < 1/2$
        (even in the finite dimensional case).
        However, our definition of the sandwiched R\'{e}nyi divergence 
        allows all $\alpha > 0$ with $\alpha \ne 1$.
        See also \cite[Remark 3.17]{Hiai21}.
    \end{enumerate}
\end{Rmk}
\begin{Lem} \label{alpha 1 renyi}
    $\Q{\psi}{\varphi}{\alpha}{1} = \renyiQ{\psi}{\varphi}{\alpha}$ 
    holds for $\alpha \in (0, \infty)\setminus\{1\}$. 
    Hence, $\D{\psi}{\varphi}{\alpha}{1} = \renyiD{\psi}{\varphi}{\alpha}$ holds 
    for $\alpha \in (0, \infty)\setminus\{1\}$.
\end{Lem}
\begin{proof}
    When $0 < \alpha < 1$, it is clear from the definition of 
    $\Q{\psi}{\varphi}{\alpha}{1}$ and $\renyiQ{\psi}{\varphi}{\alpha}$. 
    When $\alpha > 1$, it immediately follows from Lemma \ref{identity lemma} 
    with $z = 1$.
\end{proof}
\begin{Rmk} \label{alpha alpha sandwich}
    $\Q{\psi}{\varphi}{\alpha}{\alpha} = \sandQ{\psi}{\varphi}{\alpha}$ holds 
    for $\alpha \in (0, \infty) \setminus \{1\}$, 
    see \cite[Lemma 9]{KatoUeda23}.
    Remark that \cite{KatoUeda23} only shows the case that $\alpha \ge 1/2$, 
    but the discussion there works even
    for $0 < \alpha < 1/2$.
\end{Rmk}
\section{Properties of the \texorpdfstring{$\alpha$-$z$-R\'{e}nyi}{alpha-z-R\'{e}nyi} divergence}
Recall that we assume that $\psi \ne 0$ whenever 
we discuss $\D{\psi}{\varphi}{\alpha}{z}$.
\subsection{Properties for \texorpdfstring{$0 < \alpha < 1$}{0 < alpha < 1}}
In this section, we will give properties of 
the $\alpha$-$z$-R\'{e}nyi divergence 
for $0 < \alpha < 1$.
\begin{Thm} \label{Thm1}
    Let $\psi$, $\varphi$, $\psi_1$, $\psi_2$, $\varphi_1$, 
    $\varphi_2 \in \M_*^+$,  
    $0 < \alpha < 1$ and $z > 0$. 
    Then we have following properties:
    \begin{enumerate}[label=\rm{(\roman*)}]
        \item \label{scaling property1} 
        (\emph{Scaling property} or \emph{Homogeneity}) 
        For any $\lambda$, $\mu \ge 0$, 
        \begin{equation*} 
            \Q{\lambda\psi}{\mu\varphi}{\alpha}{z}
            = \lambda^{\alpha}\mu^{1-\alpha}\Q{\psi}{\varphi}{\alpha}{z}, 
            \qquad\D{\lambda\psi}{\mu\varphi}{\alpha}{z}
            = \D{\psi}{\varphi}{\alpha}{z} + \log\lambda - \log \mu. 
        \end{equation*}
        In particular, the homogeneity $\Q{\lambda\psi}{\lambda\varphi}{\alpha}{z}
        = \lambda\Q{\psi}{\varphi}{\alpha}{z}$ holds 
        for any $\lambda \ge 0$.
        \item \label{generalized mean1}
        (\emph{Generalized mean property} or \emph{Additivity under direct sum})
        \begin{equation*} 
            \Q{\psi_1 \oplus \psi_2}{\varphi_1 \oplus \varphi_2}{\alpha}{z}
            = \Q{\psi_1}{\varphi_1}{\alpha}{z} 
            + \Q{\psi_2}{\varphi_2}{\alpha}{z}.
        \end{equation*}
        In other words, 
        the continuous strictly increasing function 
        $g(t) = \exp((\alpha-1)t)$ enjoys
        \begin{equation*}
            (\psi_1(1) + \psi_2(1)) 
            g(\D{\psi_1 \oplus \psi_2}{\varphi_1 \oplus \varphi_2}{\alpha}{z})
            = \psi_1(1) g(\D{\psi_1}{\varphi_1}{\alpha}{z}) 
            + \psi_2(1)g(\D{\psi_2}{\varphi_2}{\alpha}{z}).
        \end{equation*}
        \item \label{order relation1}
        (\emph{Order relations})
        If $\psi_1 \le \psi_2$ and $z \ge \alpha$, then
        $\Q{\psi_1}{\varphi}{\alpha}{z} \le \Q{\psi_2}{\varphi}{\alpha}{z}$, 
        and if $\varphi_1 \le \varphi_2$ and $z \ge 1 - \alpha$, then
        $\Q{\psi}{\varphi_1}{\alpha}{z} \le \Q{\psi}{\varphi_2}{\alpha}{z}$.        
        Hence, 
        if $z \ge \max\{\alpha, 1-\alpha\}$, $\psi_1 \le \psi_2$
        and $\varphi_1 \le \varphi_2$, 
        then $\Q{\psi_1}{\varphi_1}{\alpha}{z} \le \Q{\psi_2}{\varphi_2}{\alpha}{z}$,
        and if $z \ge 1 - \alpha$ and $\varphi_1 \le \varphi_2$, 
        then $\D{\psi}{\varphi_1}{\alpha}{z} \ge \D{\psi}{\varphi_2}{\alpha}{z}$.
        \item \label{jointly continuity}
        (\emph{Jointly continuity}) 
        The map $(\psi, \varphi) \in \M_*^+ \times \M_*^+
        \longmapsto \Q{\psi}{\varphi}{\alpha}{z}$ 
        is jointly continuous in the norm topology.
        Hence, the map $(\psi, \varphi) \in (\M_*^+ \setminus \{0\}) \times \M_*^+
        \longmapsto \D{\psi}{\varphi}{\alpha}{z}$
        is jointly continuous in the norm topology too. 
        \item \label{limit of epsilon1}
        If $z \ge \max\{ \alpha, 1-\alpha \}$, then
        \begin{equation*} 
            \Q{\psi}{\varphi}{\alpha}{z}
            = \lim_{\varepsilon \searrow 0} \Q{\psi + \varepsilon \varphi}{\varphi + \varepsilon \psi}{\alpha}{z}
            = \inf_{\varepsilon > 0} \Q{\psi + \varepsilon \varphi}{\varphi + \varepsilon \psi}{\alpha}{z}.
        \end{equation*}
        \item \label{variational expression1}
        (The \emph{variational expression} of the $\alpha$-$z$-R\'{e}nyi divergence)
        \begin{equation} \label{variational expression for alpha < 1 ineq}
            \Q{\psi}{\varphi}{\alpha}{z}
            \le
            \inf_{a \in \M_{++}}\left\{ 
            \alpha \tr \left( (a^{1/2} h_{\psi}^{\alpha/z} a^{1/2})^{z/\alpha} \right)
            + (1 - \alpha) \tr \left( (a^{-1/2} h_{\varphi}^{(1-\alpha)/z} a^{-1/2})^{z/(1-\alpha)} \right)
            \right\}
        \end{equation}
        where $\M_{++}$ is the set of all positive invertible elements in $\M$.
        Furthermore, if $z \ge \max\{\alpha, 1-\alpha\}$ or
        $(\lambda^{-1}\varphi \le \psi \le \lambda \varphi$ for some
        $\lambda > 0)$ $\wedge$ $(z \ge \alpha \vee z \ge 1-\alpha)$, 
        then \eqref{variational expression for alpha < 1 ineq} becomes equality, 
        that is,
        \begin{equation} \label{variational expression for alpha < 1}
            \Q{\psi}{\varphi}{\alpha}{z}
            = \inf_{a \in \M_{++}}\left\{ 
            \alpha \tr \left( (a^{1/2} h_{\psi}^{\alpha/z} a^{1/2})^{z/\alpha} \right)
            + (1 - \alpha) \tr \left( (a^{-1/2} h_{\varphi}^{(1-\alpha)/z} a^{-1/2})^{z/(1-\alpha)} \right)
            \right\}.
        \end{equation}
        \item \label{strict positivity1}
        (\emph{Strict positivity})
        Assume that $\psi$, $\varphi \ne 0$.
        \begin{align}
            &\Q{\psi}{\varphi}{\alpha}{z} 
            \le \psi(1)^{\alpha}\varphi(1)^{1-\alpha}, \label{strict positivity ineq}\\
            &\D{\psi}{\varphi}{\alpha}{z} \ge \log\frac{\psi(1)}{\varphi(1)}. \notag
        \end{align}
        Moreover, the above inequalities become equalities 
        if $(1/\psi(1))\psi = (1/\varphi(1))\varphi$. 
        Also, if $z \ge 1/2$, then
        above inequalities become equalities only if 
        $(1/\psi(1))\psi = (1/\varphi(1))\varphi$.
        In particular, if $\psi(1) = \varphi(1) > 0$, then 
        $\D{\psi}{\varphi}{\alpha}{z} \ge 0$,
        and if $z \ge 1/2$, then
        equality holds if and only if
        $\psi = \varphi$.
        \item \label{DPI1}
        (\emph{DPI} or \emph{Monotonicity}) 
        Let $\N$ be a $\sigma$-finite von Neumann algebra and
        $\gamma\colon \N \longrightarrow \M$ be a unital normal positive map. 
        If $z \ge \max\{ \alpha, 1 - \alpha\}$, 
        then we have
        \begin{align} 
            \Q{\psi\circ\gamma}{\varphi\circ\gamma}{\alpha}{z} 
            \ge \Q{\psi}{\varphi}{\alpha}{z}, \label{DPI ineq} \\
            \D{\psi\circ\gamma}{\varphi\circ\gamma}{\alpha}{z} 
            \le \D{\psi}{\varphi}{\alpha}{z}. \notag
        \end{align}
        \item \label{jointly convexity1}
        (\emph{Jointly convexity}) 
        If $z \ge \max\{ \alpha, 1 - \alpha\}$, 
        then $(\psi, \varphi) \longmapsto \Q{\psi}{\varphi}{\alpha}{z}$ 
        is jointly concave, 
        that is, for each $\lambda \in [0, 1]$, we have
        \begin{equation*}
            \Q{\lambda \psi_1 + (1-\lambda)\psi_2}{\lambda \varphi_1 + (1-\lambda)\varphi_2}{\alpha}{z}
            \ge \lambda \Q{\psi_1}{\varphi_1}{\alpha}{z}
            + (1- \lambda) \Q{\psi_2}{\varphi_2}{\alpha}{z}.
        \end{equation*}
        Hence, if $z \ge \max\{ \alpha, 1 - \alpha\}$, then
        $(\psi, \varphi) \longmapsto \D{\psi}{\varphi}{\alpha}{z}$
        is jointly convex on 
        $\{ (\psi, \varphi) \in \M_*^+ \times \M_*^+ \ ;\  \psi(1) = \mu\}$ 
        for any fixed $\mu > 0$.
        \item \label{monotonicity1}
        (\emph{Monotonicity} for $z$) 
        If $0 < z \le z'$, 
        then we have 
        $\Q{\psi}{\varphi}{\alpha}{z} \ge \Q{\psi}{\varphi}{\alpha}{z'}$
        and $\D{\psi}{\varphi}{\alpha}{z} \le \D{\psi}{\varphi}{\alpha}{z'}$.
    \end{enumerate}
\end{Thm}
\begin{proof}
    \ref{scaling property1} 
    This immediately follows by definition. 
    \par
    \ref{generalized mean1} 
    By \cite[Remark 30]{Terp81}, 
    we have the natural identification 
    $L^p(\M \oplus \M) \simeq L^p(\M) \times L^p(\M)$ for each $p >0$, 
    and $h_{\varphi_1 \oplus \varphi_2} \simeq (h_{\varphi_1}, h_{\varphi_2})$
    and $\| (a, b)\|_p^p = \|a\|_p^p + \|b\|_p^p$.
    Remark that the discussion there works even for $0 < p < 1$.
    With this identification, 
    we immediately obtain the desired assertion.
    \par
    \ref{order relation1} 
    Assume that $\psi_1 \le \psi_2$, 
    so that $h_{\psi_1} \le h_{\psi_2}$.
    Observe that $f(t) = t^{\alpha/z}$ is an operator monotone function 
    on $[0, \infty)$ and $f \ge 0$ since $\alpha/z \in [0, 1]$.
    Thus, $h_{\psi_1}^{\alpha/z} \le h_{\psi_2}^{\alpha/z}$ 
    by \cite[Lemma B.7]{Hiai21}.
    Then, we have $h_{\varphi}^{(1-\alpha)/2z} h_{\psi_1}^{\alpha/z} 
    h_{\varphi}^{(1-\alpha)/2z} \le h_{\varphi}^{(1-\alpha)/2z} 
    h_{\psi_2}^{\alpha/z} h_{\varphi}^{(1-\alpha)/2z}$
    and therefore 
    $ \|h_{\varphi}^{(1-\alpha)/2z} h_{\psi_1}^{\alpha/z} 
    h_{\varphi}^{(1-\alpha)/2z}\|_z^z \le \| h_{\varphi}^{(1-\alpha)/2z} 
    h_{\psi_2}^{\alpha/z} h_{\varphi}^{(1-\alpha)/2z}\|_z^z$
    by Lemma \ref{lp norm order}. 
    This inequality is exactly 
    $\Q{\psi_1}{\varphi}{\alpha}{z} \le \Q{\psi_2}{\varphi}{\alpha}{z}$. 
    \par
    We have $h_{\psi}^{\alpha/2z} h_{\varphi_1}^{(1-\alpha)/z} 
    h_{\psi}^{\alpha/2z} 
    \le h_{\psi}^{\alpha/2z} h_{\varphi_2}^{(1-\alpha)/z} 
    h_{\psi}^{\alpha/2z}$
    by the same argument as above with $\varphi$ and 
    $(1 - \alpha)/z \in [0, 1]$
    if $\varphi_1 \le \varphi_2$ and $z \ge 1-\alpha$.
    Therefore, Lemmas \ref{lp norm order} and \ref{trace equality} show
    $\Q{\psi}{\varphi_1}{\alpha}{z} \le \Q{\psi}{\varphi_2}{\alpha}{z}$. 
    \par
    \ref{jointly continuity} 
    It suffices to show that 
    $(\psi, \varphi) \in \M_*^+ \times \M_*^+
    \longmapsto \Q{\psi}{\varphi}{\alpha}{z}$ is 
    jointly continuous in the norm topology.
    Let $\psi$, $\psi_n$, $\varphi$, $\varphi_n \in \M_*^+$ 
    such that $\|\psi_n - \psi \|$, $\|\varphi_n - \varphi\| \to 0$ 
    as $n \to \infty$. 
    (This topology is rewritten as $\|h_{\psi_n} - h_{\psi}\|_1$,
    $\|h_{\varphi_n} - h_{\varphi}\|_1 \to 0$ as $n \to \infty$.)
    When $z \ge 1/2$, 
    by the triangle inequality and H\"{o}lder's inequality, we have
    \begin{align*}
        &|\Q{\psi_n}{\varphi_n}{\alpha}{z}^{1/2z} 
        - \Q{\psi}{\varphi}{\alpha}{z}^{1/2z}| 
        = \left| \| h_{\psi_n}^{\alpha/2z} 
        h_{\varphi_n}^{(1-\alpha)/2z}\|_{2z} 
        - \| h_{\psi}^{\alpha/2z} h_{\varphi}^{(1-\alpha)/2z} \|_{2z} \right| \\
        &\qquad\le \| h_{\psi_n}^{\alpha/2z} h_{\varphi_n}^{(1-\alpha)/2z} 
        - h_{\psi}^{\alpha/2z} h_{\varphi}^{(1-\alpha)/2z} \|_{2z} \\
        &\qquad\le \| h_{\psi_n}^{\alpha/2z} -h_{\psi}^{\alpha/2z} \|_{2z/\alpha}
        \| h_{\varphi_n}^{(1-\alpha)/2z} \|_{2z/(1-\alpha)}
        + \| h_{\psi}^{\alpha/2z} \|_{2z/\alpha}
        \| h_{\varphi_n}^{(1-\alpha)/2z} 
        - h_{\varphi}^{(1-\alpha)/2z} \|_{2z/(1-\alpha)}. 
    \end{align*}
    When $z < 1/2$,
    by Lemma \ref{FackKosaki thm 4.9} and H\"{o}lder's inequality, we have
    \begin{align*}
        &|\Q{\psi_n}{\varphi_n}{\alpha}{z} - \Q{\psi}{\varphi}{\alpha}{z}| 
        = \left| \| h_{\psi_n}^{\alpha/2z} h_{\varphi_n}^{(1-\alpha)/2z}\|_{2z}^{2z} -
        \| h_{\psi}^{\alpha/2z} h_{\varphi}^{(1-\alpha)/2z} \|_{2z}^{2z} \right| \\
        &\qquad\le \| h_{\psi_n}^{\alpha/2z} h_{\varphi_n}^{(1-\alpha)/2z} 
        - h_{\psi}^{\alpha/2z} h_{\varphi}^{(1-\alpha)/2z} \|_{2z}^{2z} \\
        &\qquad\le \| h_{\psi_n}^{\alpha/2z} -h_{\psi}^{\alpha/2z} \|_{2z/\alpha}^{2z}
        \| h_{\varphi_n}^{(1-\alpha)/2z} \|_{2z/(1-\alpha)}^{2z}
        + \| h_{\psi}^{\alpha/2z} \|_{2z/\alpha}^{2z}
        \| h_{\varphi_n}^{(1-\alpha)/2z} 
        - h_{\varphi}^{(1-\alpha)/2z} \|_{2z/(1-\alpha)}^{2z}.
    \end{align*}
    Hence, the joint continuity holds for any $z > 0$ 
    by Lemma \ref{norm continuity}. 
    \par
    \ref{limit of epsilon1} immediately follows from 
    \ref{order relation1} and \ref{jointly continuity}.
    \par
    \ref{variational expression1} 
    This proof is similar to the proof of \cite[Lemma 3.19]{Hiai21}. 
    Firstly, we show inequality \eqref{variational expression for alpha < 1 ineq}.
    For every $a \in \M_{++}$, we have
    \begin{align}
        \Q{\psi}{\varphi}{\alpha}{z}
        &= \| h_{\psi}^{\alpha/2z} h_{\varphi}^{(1-\alpha)/2z} \|_{2z}^{2z}
        = \| h_{\psi}^{\alpha/2z} a^{1/2} a^{-1/2} 
        h_{\varphi}^{(1-\alpha)/2z} \|_{2z}^{2z} \notag\\
        &\le \| h_{\psi}^{\alpha/2z} a^{1/2} \|_{2z/\alpha}^{2z} 
        \|a^{-1/2} h_{\varphi}^{(1-\alpha)/2z} \|_{2z/(1-\alpha)}^{2z} 
        \label{holder alpha less than 1}\\
        &= \left( \tr \left( (a^{1/2} h_{\psi}^{\alpha/z} a^{1/2})^{z/\alpha} \right) \right)^{\alpha}
        \left(\tr \left( (a^{-1/2} h_{\varphi}^{(1-\alpha)/z} a^{-1/2})^{z/(1-\alpha)} \right) \right)^{1-\alpha} \notag \\
        &\le \alpha \tr \left( (a^{1/2} h_{\psi}^{\alpha/z} a^{1/2})^{z/\alpha} \right)
        + (1 - \alpha) \tr \left( (a^{-1/2} h_{\varphi}^{(1-\alpha)/z} a^{-1/2})^{z/(1-\alpha)} \right). \notag
    \end{align}
    Here, the first inequality is due to H\"{o}lder's inequality 
    with $1/2z = \alpha/2z + (1 - \alpha)/2z$ 
    and the second due to Young's inequality
    (or the weighted arithmetic mean and geometric mean inequality).
    \par
    We next show equality \eqref{variational expression for alpha < 1}.
    It suffices to prove the converse of inequality \eqref{variational expression for alpha < 1 ineq}.
    Firstly we assume that 
    $\lambda^{-1} \varphi \le \psi \le \lambda \varphi$ 
    for some $\lambda > 0$ and $z \ge \alpha$. 
    Then we may and do
    assume that $\psi$, $\varphi$ are faithful, 
    since $s(\psi) = s(\varphi)$ holds. 
    By \cite[Lemma B.7]{Hiai21},
    \begin{equation*}
        \lambda^{-\alpha/z}  h_{\varphi}^{\alpha/z} 
        \le h_{\psi}^{\alpha/z} 
        \le \lambda^{\alpha/z}  h_{\varphi}^{\alpha/z} 
    \end{equation*}
    since $\alpha/z \in (0, 1]$. 
    Then, multiplying $h_{\varphi}^{(1-\alpha)/2z}$ from both sides, 
    we have
    \begin{equation*}
        \lambda^{-\alpha/z}  h_{\varphi}^{1/z} 
        \le h_{\varphi}^{(1-\alpha)/2z} h_{\psi}^{\alpha/z} h_{\varphi}^{(1-\alpha)/2z} 
        \le \lambda^{\alpha/z}  h_{\varphi}^{1/z}.
    \end{equation*}
    By \cite[Lemma A.58]{Hiai21}, there exist $b$, $c \in \M$
    such that
    \begin{equation*}
        h_{\varphi}^{(1-\alpha)/2z} 
        = b (h_{\varphi}^{(1-\alpha)/2z} h_{\psi}^{\alpha/z} h_{\varphi}^{(1-\alpha)/2z})^{(1-\alpha)/2}, 
        \qquad (h_{\varphi}^{(1-\alpha)/2z} h_{\psi}^{\alpha/z} h_{\varphi}^{(1-\alpha)/2z})^{(1-\alpha)/2}
        = c h_{\varphi}^{(1-\alpha)/2z}.
    \end{equation*}
    We note that $0 < p = (1-\alpha)/z \le 1/z$ holds 
    by $0 < \alpha < 1$. Thus, the assumption of \cite[equation (A.24)]{Hiai21} holds.
    Clearly, $bc = cb = 1$ so that $c = b^{-1}$.
    We put $a_0 := b b^*$ thus $a_0^{-1} = c^* c$.
    Since 
    \begin{equation*}
        (h_{\varphi}^{(1-\alpha)/2z} h_{\psi}^{\alpha/z} h_{\varphi}^{(1-\alpha)/2z})^{(1-\alpha)/2} 
        b^* h_{\psi}^{\alpha/z} b 
        (h_{\varphi}^{(1-\alpha)/2z} h_{\psi}^{\alpha/z} h_{\varphi}^{(1-\alpha)/2z})^{(1-\alpha)/2}
        = h_{\varphi}^{(1-\alpha)/2z} h_{\psi}^{\alpha/z} h_{\varphi}^{(1-\alpha)/2z},
    \end{equation*}
    we have
    $b^* h_{\psi}^{\alpha/z} b 
    = (h_{\varphi}^{(1-\alpha)/2z} h_{\psi}^{\alpha/z} 
    h_{\varphi}^{(1-\alpha)/2z})^{\alpha}$.
    Thus, using $a_0 := b b^*$ and Lemma \ref{trace equality}, we obtain
    \begin{equation} \label{ve proof equality 1} 
        \tr \left( (a_0^{1/2} h_{\psi}^{\alpha/z} a_0^{1/2})^{z/\alpha} \right)
        = \tr \left( (h_{\varphi}^{(1-\alpha)/2z} h_{\psi}^{\alpha/z} h_{\varphi}^{(1-\alpha)/2z})^{z} \right). 
    \end{equation}
    Since
    \begin{equation*}
        h_{\varphi}^{(1-\alpha)/2z} a_0^{-1} h_{\varphi}^{(1-\alpha)/2z} 
        = h_{\varphi}^{(1-\alpha)/2z} c^* c h_{\varphi}^{(1-\alpha)/2z} 
        = (h_{\varphi}^{(1-\alpha)/2z} h_{\psi}^{\alpha/z} h_{\varphi}^{(1-\alpha)/2z})^{1-\alpha},
    \end{equation*}
    we similarly have
    \begin{equation} \label{ve proof equality 2}
        \tr \left( (a_0^{-1/2} h_{\varphi}^{(1-\alpha)/z} 
        a_0^{-1/2})^{z/(1-\alpha)} \right)
        = \tr \left( (h_{\varphi}^{(1-\alpha)/2z} h_{\psi}^{\alpha/z} h_{\varphi}^{(1-\alpha)/2z})^{z} \right). 
    \end{equation}
    Hence, by \eqref{ve proof equality 1}, \eqref{ve proof equality 2}, 
    \begin{align*}
        &\Q{\psi}{\varphi}{\alpha}{z} 
        = \tr \left( (h_{\varphi}^{(1-\alpha)/2z} h_{\psi}^{\alpha/z} h_{\varphi}^{(1-\alpha)/2z})^{z} \right)\\
        &\qquad= \alpha \tr \left( (h_{\varphi}^{(1-\alpha)/2z} h_{\psi}^{\alpha/z} h_{\varphi}^{(1-\alpha)/2z})^{z} \right)
        + (1-\alpha) \tr \left( (h_{\varphi}^{(1-\alpha)/2z} h_{\psi}^{\alpha/z} h_{\varphi}^{(1-\alpha)/2z})^{z} \right) \\
        &\qquad= \alpha \tr \left( (a_0^{1/2} h_{\psi}^{\alpha/z} a_0^{1/2})^{z/\alpha} \right)
        + (1-\alpha) \tr \left( (a_0^{-1/2} h_{\varphi}^{(1-\alpha)/z} a_0^{-1/2})^{z/(1-\alpha)} \right) \\
        &\qquad\ge \text{RHS of \eqref{variational expression for alpha < 1}}.
    \end{align*}
    Therefore, equality \eqref{variational expression for alpha < 1} holds 
    when $\lambda^{-1} \varphi \le \psi \le \lambda \varphi$ 
    for some $\lambda > 0$ and $z \ge \alpha$. 
    \par    
    We next consider the case of 
    $\lambda^{-1} \varphi \le \psi \le \lambda \varphi$ 
    for some $\lambda > 0$ and $z \ge 1 - \alpha$. 
    We replace $\psi$, $\varphi$ and $\alpha$ by $\varphi$, $\psi$ and $1-\alpha$, 
    then we can make the same argument as above. 
    Therefore, there exist $b'$, $c' \in \M$ such that
    \begin{equation*}
        h_{\psi}^{\alpha/2z} 
        = b' (h_{\psi}^{\alpha/2z} h_{\varphi}^{(1-\alpha)/z} h_{\psi}^{\alpha/2z})^{\alpha/2}, 
        \qquad(h_{\psi}^{\alpha/2z} h_{\varphi}^{(1-\alpha)/z} h_{\psi}^{\alpha/2z})^{\alpha/2}
        = c' h_{\psi}^{\alpha/2z}.
    \end{equation*}
    By the same calculation as above, we have
    \begin{equation*}
        h_{\psi}^{\alpha/2z}  c'^* c' h_{\psi}^{\alpha/2z} 
        = (h_{\psi}^{\alpha/2z} h_{\varphi}^{(1-\alpha)/z} h_{\psi}^{\alpha/2z})^\alpha, 
        \qquad b'^* h_{\varphi}^{(1-\alpha)/z} b' 
        = (h_{\psi}^{\alpha/2z} h_{\varphi}^{(1-\alpha)/z} h_{\psi}^{\alpha/2z})^{1-\alpha}
    \end{equation*}
    and, by putting $a_0' := c'^*c'$, 
    \begin{align*}
        &\tr \left( (a_0'^{1/2} h_{\psi}^{\alpha/z} a_0'^{1/2})^{z/\alpha} \right)
        = \tr \left( (h_{\psi}^{\alpha/2z} h_{\varphi}^{(1-\alpha)/z}h_{\psi}^{\alpha/2z})^{z} \right), \\
        &\tr \left( (a_0'^{-1/2} h_{\varphi}^{(1-\alpha)/z} a_0'^{-1/2})^{z/(1-\alpha)} \right)
        = \tr \left( (h_{\psi}^{\alpha/2z} h_{\varphi}^{(1-\alpha)/z}h_{\psi}^{\alpha/2z})^{z} \right).
    \end{align*}
    Since we can do the same calculation as above,
    equality \eqref{variational expression for alpha < 1} holds 
    in the case that $\lambda^{-1} \varphi \le \psi \le \lambda \varphi$ 
    for some $\lambda > 0$ and $z \ge 1-\alpha$. 
    \par
    Finally, we assume that $z \ge \max\{\alpha, 1-\alpha\}$.
    For any $\varepsilon > 0$, there exists $\lambda > 0$ such that 
    $\lambda^{-1}(\varphi+\varepsilon\psi)
    \le\psi + \varepsilon\varphi \le \lambda(\varphi+\varepsilon\psi)$.  
    Then, we can use equality \eqref{variational expression for alpha < 1} 
    for the first case, 
    and by \ref{limit of epsilon1}, 
    \begin{align*}
        &\Q{\psi}{\varphi}{\alpha}{z} 
        = \inf_{\varepsilon > 0}
        \Q{\psi + \varepsilon\varphi}{\varphi+\varepsilon\psi}{\alpha}{z} \\
        &\qquad= \inf_{\varepsilon > 0} \inf_{a \in \M_{++}}\left\{ 
        \alpha \tr \left( (a^{1/2} h_{\psi+\varepsilon\varphi}^{\alpha/z} a^{1/2})^{z/\alpha} \right)
        + (1 - \alpha) \tr \left( (a^{-1/2} h_{\varphi+\varepsilon\psi}^{(1-\alpha)/z} a^{-1/2})^{z/(1-\alpha)} \right)\right\} \\
        &\qquad= \inf_{a \in \M_{++}}\inf_{\varepsilon > 0}\left\{ 
        \alpha \| a^{1/2} (h_{\psi}+\varepsilon h_{\varphi})^{\alpha/z} a^{1/2}\|_{z/\alpha}^{z/\alpha} \right. \\
        &\qquad\qquad\qquad\qquad\qquad\left.
        + (1 - \alpha) \| a^{-1/2} (h_{\varphi}+\varepsilon h_{\psi})^{(1-\alpha)/z} a^{-1/2} \|_{z/(1-\alpha)}^{z/(1-\alpha)}
        \right\}.
    \end{align*}
    We do the same argument as the proof of \cite[Lemma 3.19]{Hiai21} 
    (see \cite[pp.~38--39]{Hiai21}). Since $\alpha/z$, $(1-\alpha)/z \in (0, 1)$, 
    we have
    \begin{align*}
        &\| a^{1/2} (h_{\psi}+\varepsilon h_{\varphi})^{\alpha/z} a^{1/2}\|_{z/\alpha}^{z/\alpha}
        \searrow \| a^{1/2} h_{\psi}^{\alpha/z} a^{1/2}\|_{z/\alpha}^{z/\alpha}, \\
        &\| a^{-1/2} (h_{\varphi}+\varepsilon h_{\psi})^{(1-\alpha)/z} a^{-1/2} \|_{z/(1-\alpha)}^{z/(1-\alpha)}
        \searrow \| a^{-1/2} h_{\varphi}^{(1-\alpha)/z} a^{-1/2} \|_{z/(1-\alpha)}^{z/(1-\alpha)}
        \qquad(\varepsilon \searrow 0).
    \end{align*}
    Therefore, if $z \ge \max\{\alpha, 1-\alpha\}$,
    then
    \begin{align*}
        &\Q{\psi}{\varphi}{\alpha}{z} 
        = \inf_{a \in \M_{++}}\inf_{\varepsilon > 0}\left\{ 
        \alpha \| a^{1/2} (h_{\psi}+\varepsilon h_{\varphi})^{\alpha/z} a^{1/2}\|_{z/\alpha}^{z/\alpha} \right. \\
        &\qquad\qquad\qquad\qquad\qquad\qquad\left.
        + (1 - \alpha) \| a^{-1/2} (h_{\varphi}+\varepsilon h_{\psi})^{(1-\alpha)/z} a^{-1/2} \|_{z/(1-\alpha)}^{z/(1-\alpha)}
        \right\} \\
        &\qquad= \inf_{a \in \M_{++}}\left\{ 
        \alpha \tr \left( (a^{1/2} h_{\psi}^{\alpha/z} a^{1/2})^{z/\alpha} \right)
        + (1 - \alpha) \tr \left( (a^{-1/2} h_{\varphi}^{(1-\alpha)/z} a^{-1/2})^{z/(1-\alpha)} \right)
        \right\}
    \end{align*}
    is obtained.
    \par
    \ref{strict positivity1} 
    We use inequalities 
    \eqref{holder alpha less than 1} with $a = 1$ (the unit of $\M$)
    in the proof of \ref{variational expression1}.
    We have
    \begin{equation*}
        \Q{\psi}{\varphi}{\alpha}{z} 
        \le \|h_{\psi}^{\alpha/2z}\|_{2z/\alpha}^{2z} 
        \| h_{\varphi}^{(1-\alpha)/2z} \|_{2z/(1-\alpha)}^{2z} 
        = (\tr(h_{\psi}))^{\alpha} (\tr(h_{\varphi}))^{1-\alpha}
        = \psi(1)^{\alpha}\varphi(1)^{1-\alpha}.
    \end{equation*}
    We next assume that $(1/\psi(1))\psi = (1/\varphi(1))\varphi$ holds.
    We put $\hat{\psi} := (1/\psi(1))\psi$ and 
    $\hat{\varphi} := (1/\varphi(1))\varphi$.
    By \ref{scaling property1} and
    $\Q{\psi}{\psi}{\alpha}{z} = \psi(1)$, 
    we obtain
    \begin{equation*}
        \Q{\psi}{\varphi}{\alpha}{z}
        = \Q{\psi(1)\hat{\psi}}{\varphi(1)\hat{\varphi}}{\alpha}{z}
        = \psi(1)^\alpha \varphi(1)^{1-\alpha} 
        \Q{\hat{\psi}}{\hat{\varphi}}{\alpha}{z}
        = \psi(1)^\alpha \varphi(1)^{1-\alpha}.
    \end{equation*}
    The last equality is due to the assumption that 
    $\hat{\psi} = \hat{\varphi}$ and $\hat{\psi}(1) = 1$.
    Conversely, 
    we assume that 
    $\Q{\psi}{\varphi}{\alpha}{z} 
    = \psi(1)^{\alpha}\varphi(1)^{1-\alpha}$
    and $z \ge 1/2$.
    Then, we rewrite this as 
    $\|h_{\psi}^{\alpha/2z} h_{\varphi}^{(1-\alpha)/2z} \|_{2z}^{2z}
    = \|h_{\psi}^{\alpha/2z}\|_{2z/\alpha}^{2z} 
    \| h_{\varphi}^{(1-\alpha)/2z} \|_{2z/(1-\alpha)}^{2z}$.
    Hence, we have $h_{\psi} = \lambda h_{\varphi}$ for some $\lambda > 0$ 
    by Theorem \ref{equality Holder thm} in \ref{Appendix Holder}. 
    Therefore, $(1/\psi(1))\psi = (1/\varphi(1))\varphi$ holds.
    \par
    \ref{DPI1} 
    Assume that $z \ge \max \{\alpha, 1-\alpha\}$.
    If $\psi = 0$ or $\varphi = 0$, 
    then $\Q{\psi}{\varphi}{\alpha}{z} = 0$
    and inequality \eqref{DPI ineq} clearly holds.
    Therefore, we may assume that $\psi$, $\varphi \ne 0$.
    For each $\varepsilon>0$, 
    we may and do assume that $\psi + \varepsilon\varphi$, 
    $\varphi + \varepsilon \psi$ are faithful.
    In addition, 
    we consider $\gamma_{\omega, \delta} \colon \N \longrightarrow \M$ for 
    a faithful normal state $\omega \in \N_*^+$ and $\delta \in (0, 1)$
    (n.b., $\N$ is $\sigma$-finite, and hence such an $\omega$ exists) 
    such that
    \begin{equation*}
        \gamma_{\omega, \delta}(b) 
        = (1-\delta)\gamma(b) + \delta \omega(b) 1_{\M}.
    \end{equation*}
    It is immediately seen that $\gamma_{\omega, \delta}$ 
    is a unital positive linear map
    and $\varphi \circ \gamma_{\omega, \delta}$ is faithful 
    for any $\varphi \ne 0 \in \M_*^+$.
    Because of this, 
    we can use inequality \eqref{appendix inequality} in \ref{Appendix}
    with the faithful positive linear functionals 
    $\psi+\varepsilon\varphi$, $\varphi+\varepsilon\psi$, 
    $(\psi+\varepsilon\varphi)\circ \gamma_{\omega, \delta}$, 
    $(\varphi+\varepsilon\psi)\circ \gamma_{\omega, \delta}$
    with $p = z/\alpha$, $z/ (1 -\alpha) \ge 1$. 
    We put $\psi_{\varepsilon}:= \psi+\varepsilon\varphi$
    and $\varphi_{\varepsilon}:= \varphi+\varepsilon\psi$.
    By Lemma \ref{trace equality} and inequality \eqref{appendix inequality} 
    with $\psi_{\varepsilon}$, 
    we have, for any $b \in \N_{+}$,
    \begin{align} 
        \tr \left( (b^{1/2} h_{\psi_{\varepsilon}\circ \gamma_{\omega, \delta}}^{\alpha/z} 
        b^{1/2})^{z/\alpha} \right)
        &= \| h_{\psi_{\varepsilon}\circ \gamma_{\omega, \delta}}^{\alpha/2z}
        \, b \,  h_{\psi_{\varepsilon}\circ \gamma_{\omega, \delta}}^{\alpha/2z}\|_{z/\alpha}^{z/\alpha} \notag \\
        &\ge \|h_{\psi_{\varepsilon}}^{\alpha/2z} \gamma_{\omega, \delta}(b) h_{\psi_{\varepsilon}}^{\alpha/2z} \|_{z/\alpha}^{z/\alpha} \notag \\
        &= \tr\left( (\gamma_{\omega, \delta}(b)^{1/2} h_{\psi_{\varepsilon}}^{\alpha/z} 
        \gamma_{\omega, \delta}(b)^{1/2})^{z/\alpha} \right). \label{DPI ineq 1}
    \end{align}
    Similarly, Lemma \ref{trace equality}, 
    inequality \eqref{appendix inequality} with $\varphi_{\varepsilon}$ and 
    Choi's inequality (see \cite[Corollary 2.3]{Choi74})
    imply that
    for any $b \in \N_{++}$, we have
    \begin{align}
        \tr \left( (b^{-1/2} h_{\varphi_{\varepsilon}\circ \gamma_{\omega, \delta}}^{(1 -\alpha)/z} 
        b^{-1/2})^{z/(1-\alpha)} \right)
        &= \| h_{\varphi_{\varepsilon}\circ \gamma_{\omega, \delta}}^{(1-\alpha)/2z}
        b^{-1}  h_{\varphi_{\varepsilon}\circ \gamma_{\omega, \delta}}^{(1-\alpha)/2z}\|_{z/(1-\alpha)}^{z/(1-\alpha)} \notag \\
        &\ge \|h_{\varphi_{\varepsilon}}^{(1-\alpha)/2z} \gamma_{\omega, \delta}(b^{-1}) h_{\varphi_{\varepsilon}}^{(1-\alpha)/2z} \|_{z/(1-\alpha)}^{z/(1-\alpha)} \notag \\
        &\ge \|h_{\varphi_{\varepsilon}}^{(1-\alpha)/2z} \gamma_{\omega, \delta}(b)^{-1} h_{\varphi_{\varepsilon}}^{(1-\alpha)/2z} \|_{z/(1-\alpha)}^{z/(1-\alpha)} \notag \\
        &= \tr\left( (\gamma_{\omega, \delta}(b)^{-1/2} h_{\varphi_{\varepsilon}}^{(1-\alpha)/z} 
        \gamma_{\omega, \delta}(b)^{-1/2})^{z/(1-\alpha)} \right). \label{DPI ineq 2}
    \end{align}
    Combining \eqref{DPI ineq 1}, \eqref{DPI ineq 2} and 
    the variational expression \ref{variational expression1}, 
    we have
    \begin{align*}
        &\Q{(\psi + \varepsilon\varphi)\circ \gamma_{\omega, \delta}}{(\varphi + \varepsilon\psi)\circ \gamma_{\omega, \delta}}{\alpha}{z}
        = \Q{\psi_{\varepsilon}\circ \gamma_{\omega, \delta}}{\varphi_{\varepsilon}\circ \gamma_{\omega, \delta}}{\alpha}{z} \\
        &\qquad= \inf_{b \in \N_{++}}\left\{ 
        \alpha \tr \left( (b^{1/2} h_{\psi_{\varepsilon}\circ \gamma_{\omega, \delta}}^{\alpha/z} b^{1/2})^{z/\alpha} \right)
        + (1 - \alpha) \tr \left( (a^{-1/2} h_{\varphi_{\varepsilon}\circ \gamma_{\omega, \delta}}^{(1-\alpha)/z} a^{-1/2})^{z/(1-\alpha)} \right)
        \right\} \\
        &\qquad\ge \inf_{b \in \N_{++}}\left\{ 
        \alpha \tr\left( (\gamma_{\omega, \delta}(b)^{1/2} h_{\psi_{\varepsilon}}^{\alpha/z} 
        \gamma_{\omega, \delta}(b)^{1/2})^{z/\alpha} \right) \right. \\
        &\qquad\qquad\qquad\qquad \left.
        + (1-\alpha) \tr\left( (\gamma_{\omega, \delta}(b)^{-1/2} h_{\varphi_{\varepsilon}}^{(1-\alpha)/z} 
        \gamma_{\omega, \delta}(b)^{-1/2})^{z/(1-\alpha)} \right) 
        \right\} \\
        &\qquad\ge \Q{\psi_{\varepsilon}}{\varphi_{\varepsilon}}{\alpha}{z}
        = \Q{\psi + \varepsilon\varphi}{\varphi+\varepsilon\psi}{\alpha}{z}.
    \end{align*}
    By the joint continuity \ref{jointly continuity}, 
    we obtain $\Q{\psi\circ\gamma}{\varphi\circ\gamma}{\alpha}{z} 
    \ge \Q{\psi}{\varphi}{\alpha}{z}$ 
    as $\delta \searrow 0$, $\varepsilon \searrow 0$.
    \par
    \ref{jointly convexity1}
    We can get \ref{jointly convexity1} as a collorary of \ref{DPI1}.
    In fact, we consider a unital positive map 
    $\gamma\colon \M \longrightarrow \M \oplus \M$ 
    given by $\gamma(a) := a \oplus a$ 
    and $\psi := \lambda \psi_1 \oplus (1-\lambda)\psi_2$,
    $\varphi := \lambda \varphi_1 \oplus (1-\lambda)\varphi_2 
    \in (\M \oplus \M)_*^+$.
    Then, we have $\psi \circ \gamma = \lambda \psi_1 + (1-\lambda)\psi_2$,
    $\varphi \circ \gamma= \lambda \varphi_1 + (1-\lambda)\varphi_2 
    \in \M_*^+$.
    By the DPI \ref{DPI1} with this map $\gamma$, the additivity under direct sum
    \ref{generalized mean1} and the homogeneity \ref{scaling property1}, 
    we have the joint concavity of $\Q{\psi}{\varphi}{\alpha}{z}$.
    The second assertion follows from that $(1/(\alpha -1))\log$ 
    is a monotone decreasing convex function.
    \par
    \ref{monotonicity1}
    We will use the ALT (Araki--Lieb--Thirring) inequality 
    in the setting of Haagerup non-commutative $L^p$-spaces \cite{Kosaki92}
    due to Kosaki.
    This method is the same as \cite[Remarks 3.18(1)]{Hiai21}. 
    Assume that $0 < z \le z'$. 
    We firstly observe that $\Q{\psi}{\varphi}{\alpha}{z}$ and
    $\Q{\psi}{\varphi}{\alpha}{z'}$ can be
    written as 
    \begin{equation*}
        \Q{\psi}{\varphi}{\alpha}{z}
        = \| h_{\psi}^{\alpha/2z} h_{\varphi}^{(1-\alpha)/2z} \|_{2z}^{2z}, 
        \qquad\Q{\psi}{\varphi}{\alpha}{z'}
        = \| |h_{\psi}^{\alpha/2z'} h_{\varphi}^{(1-\alpha)/2z'}|^{z'/z} \|_{2z}^{2z}.
    \end{equation*}
    By the ALT inequality \cite[Theorem 4]{Kosaki92} 
    (the case that $r = z'/z$, $p_1 = 2z/\alpha$ and $p_2 = 2z/(1-\alpha)$), 
    we obtain
    \begin{equation*}
        \| |h_{\psi}^{\alpha/2z'} h_{\varphi}^{(1-\alpha)/2z'}|^{z'/z} \|_{2z}
        \le \| h_{\psi}^{\alpha/2z} h_{\varphi}^{(1-\alpha)/2z} \|_{2z},
    \end{equation*}
    which yields the desired inequality for $\Q{\psi}{\varphi}{\alpha}{z}$.
\end{proof}
\begin{Rmk}
    We need only the positivity of $\gamma$ to prove the DPI 
    when $0 < \alpha <1$.
    Since $z = 1$ is included in $z \ge \max\{\alpha, 1-\alpha\}$, 
    we show the DPI for the R\'{e}nyi divergence with 
    unital normal positive map $\gamma$.
    (However, the previous results require that 
    $\gamma$ is a Schwarz map to prove the DPI 
    for the R\'{e}nyi divergence, cf.~\cite[Theorem 3.2(8)]{Hiai21}.)
\end{Rmk}
\begin{Rmk} \label{rmk of monotonicity}
    As a corollary of \ref{monotonicity1}, 
    we obtain a relation between the R\'{e}nyi divergence 
    $\renyiD{\psi}{\varphi}{\alpha}$
    and the sandwiched R\'{e}nyi divergence $\sandD{\psi}{\varphi}{\alpha}$.
    By Lemma \ref{alpha 1 renyi} and Remark \ref{alpha alpha sandwich}, 
    we have, for $0 < \alpha < 1$, 
    \begin{equation} \label{renyi sandwich renyi}
        \sandD{\psi}{\varphi}{\alpha} 
        = \D{\psi}{\varphi}{\alpha}{\alpha}
        \le \D{\psi}{\varphi}{\alpha}{1}
        = \renyiD{\psi}{\varphi}{\alpha}.
    \end{equation}
    If Question \ref{question monotonicity} in subsection \ref{questions} 
    was settled,
    then inequality \eqref{renyi sandwich renyi} would hold
    for any $\alpha \in (0, \infty)\setminus\{1\}$.
    Inequality \eqref{renyi sandwich renyi} was proved for $\alpha \ge 1/2$
    in \cite[Theorem 12]{BST18}, \cite[Corollary 3.6]{Jencova18}, 
    \cite[Theorem 3.3]{Jencova21}, see also \cite[Theorem 3.16(6)]{Hiai21}.
\end{Rmk}
\subsection{Properties for \texorpdfstring{$\alpha > 1$}{alpha > 1}}
In this section, we will prove some properties of 
the $\alpha$-$z$-R\'{e}nyi divergence 
for $\alpha > 1$.
The numbering of properties are the same as in Theorem \ref{Thm1}.
\begin{Thm}\label{Thm2}
    Let $\psi$, $\varphi$, $\psi_1$, $\psi_2$, $\varphi_1$, 
    $\varphi_2 \in \M_*^+$, 
    $\alpha > 1$ and $z > 0$. 
    Then we have following properties:
    \begin{enumerate}[label=\rm{(\roman*)}]
        \item \label{scaling property2}
        (\emph{Scaling property} or \emph{Homogeneity}) 
        For $\lambda$, $\mu \ge 0$, 
        \begin{align} 
            &\Q{\lambda\psi}{\mu\varphi}{\alpha}{z}
            = \lambda^{\alpha}\mu^{1-\alpha}\Q{\psi}{\varphi}{\alpha}{z}, 
            \label{scaling Q2} \\ 
            &\D{\lambda\psi}{\mu\varphi}{\alpha}{z}
            = \D{\psi}{\varphi}{\alpha}{z} + \log\lambda - \log \mu. 
            \label{scaling D2}
        \end{align}
        except when $\mu = 0$ and $\lambda\psi \ne 0$ for the first equality.\footnote{
        In this case, $\text{LHS of \eqref{scaling Q2}}  = \infty$ and 
        $\text{RHS of \eqref{scaling Q2}}  = 0$ 
        if we use the convention $0 \cdot \infty = 0$.
        However, both the sides of \eqref{scaling D2}
        equal $\infty$ when $\mu = 0$. \label{footnote scaling}}
        The homogeneity $\Q{\lambda\psi}{\lambda\varphi}{\alpha}{z}
        = \lambda\Q{\psi}{\varphi}{\alpha}{z}$ holds 
        for $\lambda \ge 0$.
        \item \label{generalized mean2}
        (\emph{Generalized mean property} or \emph{Additivity under direct sum})
        \begin{equation*} 
            \Q{\psi_1 \oplus \psi_2}{\varphi_1 \oplus \varphi_2}{\alpha}{z}
            = \Q{\psi_1}{\varphi_1}{\alpha}{z} 
            + \Q{\psi_2}{\varphi_2}{\alpha}{z}.
        \end{equation*}
        In other words, 
        the continuous strictly increasing function 
        $g(t) = \exp((\alpha-1)t)$ enjoys
        \begin{equation*}
            (\psi_1(1) + \psi_2(1)) 
            g(\D{\psi_1 \oplus \psi_2}{\varphi_1 \oplus \varphi_2}{\alpha}{z})
            = \psi_1(1) g(\D{\psi_1}{\varphi_1}{\alpha}{z}) 
            + \psi_2(1)g(\D{\psi_2}{\varphi_2}{\alpha}{z}).
        \end{equation*}
        \item \label{order relation2}
        (\emph{Order relations})
        If $\psi_1 \le \psi_2$ and $z \ge \alpha$, then
        $\Q{\psi_1}{\varphi}{\alpha}{z} \le \Q{\psi_2}{\varphi}{\alpha}{z}$,
        and if $\varphi_1 \le \varphi_2$ and $z \ge \alpha - 1$, then
        $\Q{\psi}{\varphi_1}{\alpha}{z} \ge \Q{\psi}{\varphi_2}{\alpha}{z}$.
        Hence, if $z \ge \alpha$, $\psi_1 \le \psi_2$ 
        and $\varphi_1 \ge \varphi_2$, 
        then $\Q{\psi_1}{\varphi_1}{\alpha}{z} \le \Q{\psi_2}{\varphi_2}{\alpha}{z}$,
        and if $z \ge \alpha - 1$ and $\varphi_1 \le \varphi_2$, 
        then $\D{\psi}{\varphi_1}{\alpha}{z} \ge \D{\psi}{\varphi_2}{\alpha}{z}$.
        \item \label{jointly lower semicontinuity}
        (\emph{Jointly lower semi-continuity}) 
        If $z \ge \alpha/2$, then
        the map $(\psi, \varphi) \in \M_*^+ \times \M_*^+
        \longmapsto \Q{\psi}{\varphi}{\alpha}{z}$ 
        is jointly lower semi-continuous in the norm topology.
        Hence, if $z \ge \alpha/2$, then, 
        the map $(\psi, \varphi) \in (\M_*^+ \setminus \{0\}) \times \M_*^+
        \longmapsto \D{\psi}{\varphi}{\alpha}{z}$
        is jointly lower semi-continuous in the norm topology too. 
        \item \label{limit of epsilon2}
        If $z \ge \max \{\alpha -1, \alpha/2 \}$, then
        \begin{equation*} 
            \Q{\psi}{\varphi}{\alpha}{z}        
            = \lim_{\varepsilon \searrow 0} \Q{\psi}{\varphi + \varepsilon \psi}{\alpha}{z}
            = \sup_{\varepsilon > 0} \Q{\psi}{\varphi + \varepsilon \psi}{\alpha}{z}.
        \end{equation*}
        \item \label{variational expression2}
        (Variational lower estimate)
        \begin{equation} \label{variational expression for alpha > 1 ineq}
            \Q{\psi}{\varphi}{\alpha}{z}
            \ge \sup_{a \in \M_{+}}\left\{ 
            \alpha \tr \left( (a^{1/2} h_{\psi}^{\alpha/z} a^{1/2})^{z/\alpha} \right)
            - (\alpha - 1) \tr \left( (a^{1/2} h_{\varphi}^{(\alpha-1)/z} a^{1/2})^{z/(\alpha-1)} \right)
            \right\}.
        \end{equation}
    \item \label{strict positivity2}
    (\emph{Strict positivity})
    Assume that $\psi$, $\varphi \ne 0$.
    \begin{equation} \label{strict positivity ineq2}
        \Q{\psi}{\varphi}{\alpha}{z} 
        \ge \psi(1)^{\alpha}\varphi(1)^{1-\alpha}, 
        \qquad\D{\psi}{\varphi}{\alpha}{z} \ge \log\frac{\psi(1)}{\varphi(1)}.
    \end{equation}
    Moreover, the above inequalities become equalities 
    if $(1/\psi(1))\psi = (1/\varphi(1))\varphi$. 
    Also, if $z = 1$ or $z \ge \alpha/2$,
    then the above inequalities become equalities 
    only if $(1/\psi(1))\psi = (1/\varphi(1))\varphi$.
    In particular, if $\psi(1) = \varphi(1) > 0$, then 
    $\D{\psi}{\varphi}{\alpha}{z} \ge 0$,
    and if $z = 1$ or $z \ge \alpha/2$,
    then equality holds
    if and only if
    $\psi = \varphi$.
    \end{enumerate}
\end{Thm}
\begin{proof}
    \ref{scaling property2} 
    If $\lambda = 0$ or ($\mu = 0 \wedge \lambda\psi = 0)$, 
    then both the sides of \eqref{scaling Q2} equal $0$. 
    We thus consider the case that $\lambda$, $\mu > 0$.
    If $\Q{\psi}{\varphi}{\alpha}{z} < \infty$, 
    i.e., if there exists $x \in s(\varphi) L^z(M) s(\varphi)$ 
    such that identity \eqref{alpha z identity} holds with 
    $\psi$ and $\varphi$,
    then $x' = \lambda^{\alpha/z}\mu^{(1-\alpha)/z} x \in s(\varphi) L^z(M) s(\varphi)$
    satisfies $(\lambda h_{\psi})^{\alpha/z} 
    = (\mu h_{\varphi})^{(\alpha-1)/2z} x' 
    (\mu h_{\varphi})^{(\alpha-1)/2z}$.
    Thus, we have
    \begin{equation*}
        \|x'\|_z^z = \|\lambda^{\alpha/z}\mu^{(1-\alpha)/z} x\|_z^z
        = \lambda^{\alpha}\mu^{1-\alpha}\Q{\psi}{\varphi}{\alpha}{z} < \infty.
    \end{equation*}
    Replacing $\lambda$ and $\mu$ 
    by $1/\lambda$ and $1/\mu$, respectively, in the above argument,
    we can similarly show
    that $\Q{\psi}{\varphi}{\alpha}{z} < \infty$ and 
    equality \eqref{scaling Q2} holds
    if $\Q{\lambda\psi}{\mu\varphi}{\alpha}{z} < \infty$.
    Equality \eqref{scaling D2} follows by definition.
    \par
    \ref{generalized mean2}
    We can see, via the identification mentioned in the proof of 
    Theorem \ref{Thm1}\ref{generalized mean1}, that 
    there exists 
    $x \in s(\varphi_1 \oplus \varphi_2) 
    L^z(\M \oplus \M) s(\varphi_1 \oplus \varphi_2)$
    such that identity \eqref{alpha z identity} holds with
    $h_{\psi_1 \oplus \psi_2}$ and
    $h_{\varphi_1 \oplus \varphi_2}$ 
    if and only if 
    there exists
    $x_i \in s(\varphi_i) L^z(\M) s(\varphi_i)$ such
    that identity \eqref{alpha z identity} holds with 
    $h_{\psi_i}$ and
    $h_{\varphi_i}$ for each $i=1, 2$.
    \par
    \ref{order relation2}
    Assume that $\psi_1 \le \psi_2$.
    If $\Q{\psi_2}{\varphi}{\alpha}{z} = \infty$, 
    then the desired inequality clearly holds.
    Thus, we may assume that $\Q{\psi_2}{\varphi}{\alpha}{z} < \infty$ in addition.
    Then there exists $y \in L^{2z}(\M)s(\varphi)$ such that 
    $h_{\psi_2}^{\alpha/2z} = 
    y h_{\varphi}^{(\alpha-1)/2z}$ holds 
    by the definition of $\Q{\psi}{\varphi}{\alpha}{z}$ 
    and Lemma \ref{identity lemma}.
    On the other hand, when $z\ge \alpha$, 
    there exists $a \in \M$ with $\|a\| \le 1$ such that 
    $h_{\psi_1}^{\alpha/2z} = a h_{\psi_2}^{\alpha/2z}$ holds
    since $h_{\psi_1} \le h_{\psi_2}$ by \cite[Lemma A.58]{Hiai21} 
    (or \cite[Lemma A.24]{Hiai21}).
    Hence, there exists $ay \in L^{2z}s(\varphi)$
    such that
    $h_{\psi_1}^{\alpha/2z} = (ay) h_{\varphi}^{(\alpha-1)/2z}$.
    Therefore, by H\"{o}lder's inequality, we have
    $\Q{\psi_1}{\varphi}{\alpha}{z} < \infty$ and 
    $\Q{\psi_1}{\varphi}{\alpha}{z} = \|ay\|_{2z}^{2z} 
    \le \|a\|^{2z}\|y\|_{2z}^{2z} \le \Q{\psi_2}{\varphi}{\alpha}{z}$.
    \par
    We then assume that $\varphi_1 \le \varphi_2$ 
    and $\Q{\psi}{\varphi_1}{\alpha}{z} < \infty$.
    Then there exists $y \in L^{2z}(\M)s(\varphi_1)$ such that 
    $h_{\psi}^{\alpha/2z} = 
    y h_{\varphi_1}^{(\alpha-1)/2z}$ holds.
    Moreover, when $z \ge \alpha - 1$, 
    there exists $b \in s(\varphi_2)\M s(\varphi_2)$ with $\|b\| \le 1$
    such that $h_{\varphi_1}^{(\alpha-1)/2z} = b h_{\varphi_2}^{(\alpha-1)/2z}$.
    Thus, $h_{\psi}^{\alpha/2z} = yb h_{\varphi_2}^{(\alpha-1)/2z}$ holds.
    As above  
    $\Q{\psi}{\varphi_2}{\alpha}{z} = \|yb\|_{2z}^{2z} 
    \le \|y\|_{2z}^{2z}\|b\|^{2z} \le \Q{\psi}{\varphi_1}{\alpha}{z}$
    by H\"{o}lder's inequality.
    \par
    \ref{jointly lower semicontinuity}
    The proof is the same as \cite[Proposition 3.10]{Jencova18}. 
    We assume that $z \ge \alpha /2$.
    To prove the joint lower semi-continuity of $(\psi, \varphi) 
    \longmapsto \Q{\psi}{\varphi}{\alpha}{z}$ 
    is equivalent to proving that
    the set 
    $L_\lambda := \{ (\psi, \varphi) \in \M_*^+ \times \M_*^+ \ ; \ 
    \Q{\psi}{\varphi}{\alpha}{z} \le \lambda \}$ 
    is closed for every $\lambda \ge 0$. 
    Let $\psi$, $\varphi\in \M_*^+$ and $(\psi_n, \varphi_n) \in L_\lambda$
    with $\|\psi_n - \psi\| \to 0$, $\|\varphi_n - \varphi\| \to 0$ as $n \to \infty$.
    Since $(\psi_n, \varphi_n) \in L_\lambda$, 
    $\Q{\psi_n}{\varphi_n}{\alpha}{z} \le \lambda < \infty$ holds,
    then there exists $y_n \in  L^{2z}(\M) s(\varphi_n)$ such that 
    identity \eqref{alpha z identity prime} holds with $\psi_n$ and $\varphi_n$
    and $\|y_n \|_{2z}= \Q{\psi_n}{\varphi_n}{\alpha}{z}^{1/2z} \le \lambda^{1/2z}$.
    By assumption, $2z \ge \alpha > 1$ holds.
    Hence, $L^{2z}(\M)$ is a reflexive Banach space 
    and $L^{2z/(2z-1)}(\M)$ is the dual space of $L^{2z}(\M)$.
    Thus, there exists $y \in L^{2z}(\M)$ with $\|y\|_{2z} \le \lambda^{1/2z}$
    such that $y_{n_i} \to y$ as $i \to \infty$
    in the weak topology.
    (This fact is well known,
    see \cite[Chapter V, Theorem 4.2 and 13.1]{Conway90} etc.)
    We may and do replace $n_i$ by $n$ and assume that $y_n \to y$ 
    in the weak topology.
    \par
    Let $\beta := 2z/(2z - \alpha) = (1 - \alpha/2z)^{-1}$.
    By the triangle inequality and H\"{o}lder's inequality, 
    for any $a \in L^{\beta}(\M)$, we have
    \begin{align*}
        &|\tr(a (y_n h_{\varphi_n}^{(\alpha-1)/2z} - y h_{\varphi}^{(\alpha-1)/2z}))| \\
        &\qquad\le |\tr(a y_n (h_{\varphi_n}^{(\alpha-1)/2z} - h_{\varphi}^{(\alpha-1)/2z}))|
        + |\tr(a (y_n - y) h_{\varphi}^{(\alpha-1)/2z})| \\
        &\qquad\le \|a\|_{\beta}\| \|y_n\|_{2z} 
        \|h_{\varphi_n}^{(\alpha-1)/2z} - h_{\varphi}^{(\alpha-1)/2z}\|_{2z/(\alpha - 1)}
        + |\tr ((y_n - y) h_{\varphi}^{(\alpha-1)/2z} a)|.
    \end{align*}
    By Lemma \ref{norm continuity}, we have
    $\|h_{\varphi_n}^{(\alpha-1)/2z} 
    - h_{\varphi}^{(\alpha-1)/2z}\|_{2z/(\alpha - 1)} \to 0$ 
    as $n \to \infty$,
    and by $h_{\varphi}^{(\alpha-1)/2z} a \in L^{2z/(2z -1)}(\M)$, 
    we have
    $|\tr ((y_n - y) h_{\varphi}^{(\alpha-1)/2z} a)| \to 0$ as $n \to \infty$. 
    Hence we obtain that $h_{\psi_n}^{\alpha/2z} = y_n h_{\varphi_n}^{(\alpha-1)/2z} 
    \to y h_{\varphi}^{(\alpha-1)/2z}$ as $n \to \infty$ in the weak topology.
    \par
    On the other hand, $h_{\psi_n}^{\alpha/2z} \to h_{\psi}^{\alpha/2z}$ 
    in the norm topology by assumption.
    Hence, $h_{\psi}^{\alpha/2z} = y h_{\varphi}^{(\alpha-1)/2z}$ holds
    and $\|y\|_{2z} \le \lambda^{1/2z}$.
    Therefore, $L_{\lambda}$ is closed.
    \par
    \ref{limit of epsilon2} immediately follows from \ref{order relation1}
    and \ref{jointly lower semicontinuity}.
    \par
    \ref{variational expression2}
    We calculate similarly to the case of $0 < \alpha <1$.
    When $\alpha > 1$, 
    inequality \eqref{variational expression for alpha > 1 ineq} clearly holds 
    if $\Q{\psi}{\varphi}{\alpha}{z} = \infty$.
    So, we may assume that $\Q{\psi}{\varphi}{\alpha}{z} < \infty$,
    that is, there exists $y \in L^{2z}(\M)s(\varphi)$
    such that identity \eqref{alpha z identity prime} holds
    by Lemma \ref{identity lemma}.
    Then, for every $a \in \M_+$, we have
    \begin{align}
        \tr\left( (a^{1/2} h_{\psi}^{\alpha/z} a^{1/2})^{z/\alpha} \right) 
        &= \|h_{\psi}^{\alpha/2z} a^{1/2} \|_{2z/\alpha}^{2z/\alpha} 
        = \|y h_{\varphi}^{(\alpha - 1)/2z} a^{1/2}\|_{2z/\alpha}^{2z/\alpha} \notag \\
        &\le \|y\|_{2z}^{2z/\alpha} 
        \| h_{\varphi}^{(\alpha - 1)/2z} a^{1/2}\|_{2z/(\alpha - 1)}^{2z/\alpha} 
        \label{holder alpha greater than 1}\\ 
        &= \Q{\psi}{\varphi}{\alpha}{z}^{1/\alpha} 
        \left( \tr \left( 
        (a^{1/2} h_{\varphi}^{(\alpha - 1)/z} a^{1/2})^{z/(\alpha - 1)}
        \right) \right)^{(\alpha - 1)/\alpha} \notag\\
        &\le \frac{1}{\alpha} \Q{\psi}{\varphi}{\alpha}{z}
        + \frac{\alpha - 1}{\alpha} \tr \left( (a^{1/2} h_{\varphi}^{(\alpha - 1)/z} a^{1/2})^{z/(\alpha - 1)} \right), \notag
    \end{align}
    where the second equality is due to identity \eqref{alpha z identity prime}. 
    \par
    \ref{strict positivity2}
    We use inequalities \eqref{holder alpha greater than 1}.
    We put $a = 1$ 
    for \eqref{holder alpha greater than 1}, 
    and then we have
    \begin{equation*}
        \tr(h_{\psi}) =
        \| h_{\psi}^{\alpha/2z} \|_{2z/\alpha}^{2z/\alpha}
        \le \|y\|_{2z}^{2z/\alpha} 
        \| h_{\varphi}^{(\alpha - 1)/2z} \|_{2z/(\alpha - 1)}^{2z/\alpha} 
        = \Q{\psi}{\varphi}{\alpha}{z}^{1/\alpha} (\tr(h_{\varphi}))^{(\alpha - 1)/\alpha}, 
    \end{equation*}
    which yields $\Q{\psi}{\varphi}{\alpha}{z} 
    \ge \psi(1)^{\alpha}\varphi(1)^{1-\alpha}$.
    In the same way as the case of $0 < \alpha < 1$
    we can prove that 
    if $(1/\psi(1))\psi = (1/\varphi(1))\varphi$, then 
    inequalities \eqref{strict positivity ineq2} become equalities. 
    Conversely, we assume that $\Q{\psi}{\varphi}{\alpha}{z} 
    = \psi(1)^\alpha \varphi(1)^{1-\alpha}$. 
    If $z = 1$,
    then $\Q{\psi}{\varphi}{\alpha}{z}$ becomes 
    the R\'{e}nyi divergence.
    Hence, $(1/\psi(1))\psi = (1/\varphi(1))\varphi$ holds, 
    see \cite[Theorem 3.2(9)]{Hiai21}. 
    We next consider the case of $z \ge \alpha/2$. 
    Since $\Q{\psi}{\varphi}{\alpha}{z} 
    = \psi(1)^{\alpha}\varphi(1)^{1-\alpha} < \infty$, 
    there exists $x \in (s(\varphi)L^{z}(\M)s(\varphi))_+$ such that 
    identity \eqref{alpha z identity} holds. 
    By identity \eqref{alpha z identity} and the assumption,
    we observe that 
    \begin{equation*}
        \|h_{\psi}^{\alpha/z}\|_{z/\alpha}^{z} 
        = \|h_{\varphi}^{(\alpha-1)/2z} x 
        h_{\varphi}^{(\alpha-1)/2z}\|_{z/\alpha}^{z}  
        = \|x^{1/2} h_{\varphi}^{(\alpha-1)/2z}\|_{2z/\alpha}^{2z} 
        = \|x^{1/2}\|_{2z}^{2z}
        \|h_{\varphi}^{(\alpha-1)/2z}\|_{2z/(\alpha-1)}^{2z}. 
    \end{equation*} 
    Hence, we have $x^z = \lambda h_{\varphi}$ for some $\lambda > 0$ by
    Theorem \ref{equality Holder thm} in \ref{Appendix Holder}
    and thus $h_{\psi} = \mu h_{\varphi}$ for some $\mu> 0$.
    Therefore, we have $(1/\psi(1))\psi = (1/\varphi(1))\varphi$.
\end{proof}
As a corollary of \ref{order relation1} in Theorems \ref{Thm1} and \ref{Thm2}, 
we have the order axiom (see \cite[axiom (IV), Proposition 1]{AudenaertDatta15}) 
of $\D{\psi}{\varphi}{\alpha}{z}$.
\begin{Cor} \label{order axiom}
    $\D{\psi}{\varphi}{\alpha}{z}$ satisfies the order axiom 
    when $z \ge |\alpha -1|$.
\end{Cor}
\begin{proof}[Proof of Corollary \ref{order axiom}]
    It is immediately seen that 
    $\Q{\psi}{\psi}{\alpha}{z} = \psi(1)$, 
    so that $\D{\psi}{\psi}{\alpha}{z} = 0$. 
    (This claim is included in \ref{strict positivity1}.)
    Thus, if $\psi \le (\text{resp.} \ge)\  \varphi$, 
    then we have $0 = \D{\psi}{\psi}{\alpha}{z} \ge 
    (\text{resp.} \le) \ \D{\psi}{\varphi}{\alpha}{z}$.
\end{proof}
The $\alpha$-$z$-R\'{e}nyi divergence has the additivity under tensor products 
stated in \cite[Proposition 10]{KatoUeda23} 
when $\alpha < 1$ or both $\Q{\psi_i}{\varphi_i}{\alpha}{z} < \infty$ 
or $z= \alpha \in [1/2, \infty)\setminus\{1\}$. 
We will improve it below. 
\begin{Prop} \label{tensor product alpha z}
    Let $\alpha$, $z > 0$ with $\alpha \ne 1$ 
    and $\psi_i$, $\varphi_i \in \M_*^+$ with $i = 1, 2$.
    If $0 < \alpha < 1$ or both $\Q{\psi_i}{\varphi_i}{\alpha}{z} < \infty$
    or $z \ge \max\{\alpha-1, \alpha/2\}$, 
    then we have
    \begin{align} 
        &\Q{\psi_1 \bar{\otimes} \psi_2}{\varphi_1 \bar{\otimes}\varphi_2}{\alpha}{z}
        = \Q{\psi_1}{\varphi_1}{\alpha}{z} \Q{\psi_2}{\varphi_2}{\alpha}{z}, 
        \label{multiplicativity under tensor product} \\
        &\D{\psi_1 \bar{\otimes} \psi_2}{\varphi_1 \bar{\otimes}\varphi_2}{\alpha}{z}
        = \D{\psi_1}{\varphi_1}{\alpha}{z} + \D{\psi_2}{\varphi_2}{\alpha}{z}.
        \label{additivity under tensor product}
    \end{align}
\end{Prop}
\begin{proof}
    See \cite[Proposition 10]{KatoUeda23} for the case 
    that $0 < \alpha < 1$ or both $\Q{\psi_i}{\varphi_i}{\alpha}{z} < \infty$.
    \par
    We show equality \eqref{multiplicativity under tensor product} when 
    $\alpha > 1$ and $z \ge \max\{\alpha-1, \alpha/2\}$.
    For any $\varepsilon > 0$ and $i = 1, 2$,
    we put $\varphi_{i, \varepsilon} := \varphi_i + \varepsilon \psi_i$.
    There exist $\lambda_i > 0$, $i = 1, 2$
    such that
    $\psi_i \le \lambda_i \varphi_{i, \varepsilon}$ hold.
    By \cite[Lemma A.58]{Hiai21}, there exist
    $a_i \in \M$ such that
    $h_{\psi_i}^{(\alpha-1)/2z} 
    = a_i h_{\varphi_{i, \varepsilon}}^{(\alpha-1)/2z}$.
    Hence, for each $i = 1, 2$, $x_i= a_i^* h_{\psi_i}^{1/z}a_i 
    \in s(\varphi_{i, \varepsilon}) L^{z}(\M) s(\varphi_{i, \varepsilon})$
    and 
    \begin{equation*}
        h_{\psi_i}^{\alpha/z} 
        = h_{\psi_i}^{(\alpha-1)/2z} h_{\psi_i}^{1/z} 
        h_{\psi_i}^{(\alpha-1)/2z} 
        = h_{\varphi_{i, \varepsilon}}^{(\alpha-1)/2z} a_i^* h_{\psi_i}^{1/z} 
        a_i h_{\varphi_{i, \varepsilon}}^{(\alpha-1)/2z}
        = 
        h_{\varphi_{i, \varepsilon}}^{(\alpha-1)/2z}
        x_i h_{\varphi_{i, \varepsilon}}^{(\alpha-1)/2z}.
    \end{equation*}
    Thus, $x_i$ satisfies identity \eqref{alpha z identity} with
    $\psi_i$ and $\varphi_{i, \varepsilon}$
    and
    $\Q{\psi_i}{\varphi_{i, \varepsilon}}{\alpha}{z} < \infty$, 
    for each $i = 1, 2$. 
    Therefore, we can use equality \eqref{multiplicativity under tensor product}
    with $\psi_i$ and $\varphi_{i, \varepsilon}$.
    By Theorem \ref{Thm2}\ref{limit of epsilon2}, 
    $\Q{\psi_i}{\varphi_{i, \varepsilon}}{\alpha}{z} \to 
    \Q{\psi_i}{\varphi_i}{\alpha}{z}$ as $\varepsilon \searrow 0$
    and 
    $\Q{\psi_1}{\varphi_{1, \varepsilon}}{\alpha}{z} 
    \Q{\psi_2}{\varphi_{2, \varepsilon}}{\alpha}{z} \to 
    \Q{\psi_1}{\varphi_1}{\alpha}{z} \Q{\psi_2}{\varphi_2}{\alpha}{z}$ 
    as $\varepsilon \searrow 0$. 
    Since
    $\varphi_1 \bar{\otimes} \varphi_2 
    \le \varphi_{1, \varepsilon} \bar{\otimes} \varphi_{2, \varepsilon}$
    (see e.g., \cite[Section 8.8]{Stratila81})
    and $\varphi_{1, \varepsilon} \bar{\otimes} \varphi_{2, \varepsilon}
    \to \varphi_1 \bar{\otimes} \varphi_2$
    as $\varepsilon \searrow 0$ in the norm topology,
    we use \ref{order relation2} and \ref{jointly lower semicontinuity}
    of Theorem \ref{Thm2} and have
    $\Q{\psi_1 \bar{\otimes} \psi_2}{\varphi_{1, \varepsilon} \bar{\otimes}\varphi_{2, \varepsilon}}{\alpha}{z}
    \to \Q{\psi_1 \bar{\otimes} \psi_2}{\varphi_1 \bar{\otimes}\varphi_2}{\alpha}{z}$ as $\varepsilon \searrow 0$.
    Hence, we obtain equality \eqref{multiplicativity under tensor product}.
    Equality \eqref{additivity under tensor product} immediately follows  
    by the definition of $\D{\psi}{\varphi}{\alpha}{z}$.
\end{proof}
\subsection{Questions and Comments} \label{questions}
The statements of Theorem \ref{Thm2} are unfortunately not as complete as 
those of Theorem \ref{Thm1} are.
Hence we do give several comments on what we could not establish. 
\par
We could not establish the joint lower semi-continuity for all $(\alpha,z)$, 
and thus we would like to pose the following question: 
\begin{Ques} \label{question jointly lsc}
    Does our $\alpha$-$z$-R\'{e}nyi divergence satisfy the 
    joint lower semi-continuity for any $\alpha$, $z >0$ with $\alpha \ne 1$?    
\end{Ques}
It is known that this is in the affirmative in the finite dimensional case 
(see e.g., \cite[axiom (I)]{AudenaertDatta15}).
\begin{Cmt}
    Keep the notations in the proof of 
    Theorem \ref{Thm2}\ref{jointly lower semicontinuity}. 
    It seems difficult to show that 
    $y_n$ converges to some $y$ that satisfies identity \eqref{alpha z identity prime}
    with $\psi,\varphi$. 
    In fact, it is natural to use the $L^p$-$L^q$-duality 
    to prove the existence of a limit of $\{y_n\}$. 
    Then, we have to show that
    $h_{\psi_n}^{\alpha/2z}$ converges to
    $y h_{\varphi}^{(\alpha-1)/2z}$ in the weak sense.
    Here, we need $2z/\alpha \ge 1$.
\end{Cmt}
\begin{Rmk}
    If Question \ref{question jointly lsc} was settled in the affirmative, 
    then one would get rid of $z \ge \alpha/2$ from the assumption 
    of Theorem \ref{Thm2}\ref{limit of epsilon2}, 
    and hence apply the same omission to Proposition \ref{tensor product alpha z} too.
\end{Rmk}
We could unfortunately prove only a variational lower estimate 
in Theorem \ref{Thm2}\ref{variational expression2}. 
Thus the following is a natural question: 
\begin{Ques} \label{question variational expression}
    What is the condition admitting that inequality 
    \eqref{variational expression for alpha > 1 ineq} becomes equality?
\end{Ques}
In the finite dimensional case, 
Zhang \cite[Theorem 3.3]{Zhang20} established the variational expression 
of the $\alpha$-$z$-R\'{e}nyi divergence 
for any $\alpha$, $z >0$ with $\alpha \ne 1$.
Mosonyi \cite[Lemma 3.23]{Mosonyi23} did it for any $\alpha > 1$, $z >0$
in the infinite dimensional type I case under a few assumptions. 
It is desirable to show that
inequality \eqref{variational expression for alpha > 1 ineq} becomes equality
at least when $\max\{\alpha-1, \alpha/2\} \le z \le \alpha$. 
In fact, this restriction on $(\alpha,z)$ makes no trouble to prove the DPI
(see also Question \ref{question DPI}). 
The case of $z = \alpha$ (i.e., the sandwiched R\'{e}nyi divergence) was
completely settled by Jen\v{c}ov\'{a} \cite{Jencova21}.
\begin{Cmt}
    An element $a_0 \in \M_+$ which attains
    the maximum (resp.~minimum) of the variational expression of
    $\Q{\psi}{\varphi}{\alpha}{z}$ when $\alpha > 1$ (or $0<\alpha<1$)
    should be given by  
        \begin{equation*}
            a_0 = h_{\varphi}^{(1-\alpha)/2z}(h_{\varphi}^{(\alpha-1)/2z}
            h_{\psi}^{-\alpha/z}
            h_{\varphi}^{(\alpha-1)/2z})^{1-\alpha}h_{\varphi}^{(1-\alpha)/2z}
            = h_{\psi}^{-\alpha/2z}(h_{\psi}^{\alpha/2z}
            h_{\varphi}^{(1-\alpha)/z}
            h_{\psi}^{\alpha/2z})^\alpha h_{\psi}^{-\alpha/2z}
        \end{equation*}
    in a formal sense.
    (The second equality is due to Kubo--Ando's weighted geometric mean formula 
    $h_{\psi}^{-\alpha/z} \mathbin{\#}_{\alpha} h_{\varphi}^{(1-\alpha)/z}
    = h_{\varphi}^{(1-\alpha)/z} \mathbin{\#}_{1-\alpha} h_{\psi}^{-\alpha/z}$
    in a formal sense.)
    In fact, the $a_0$ and $a_0'$ in the proof of 
    Theorem \ref{Thm1}\ref{variational expression1} can be written as 
    the second term above in a formal sense 
    and the last term above in a formal sense, respectively.
    Besides, if 
    \eqref{variational expression for alpha > 1 ineq} became 
    equality when $\psi \le \lambda \varphi$ for some $\lambda > 0$, 
    then one would be able to show it for any $\psi$, $\varphi$ 
    under the same assumption as Theorem \ref{Thm2}\ref{limit of epsilon2} 
    (by the same argument as the case of $0 < \alpha < 1$).
\end{Cmt}
We could not show the sufficiency for equality of the strict positivity 
for any $\alpha$, $z > 0$ with $\alpha \ne 1$
in Theorems \ref{Thm1} and \ref{Thm2}\ref{strict positivity1}.
\begin{Ques}
    Let $\alpha$, $z > 0$ with $\alpha \ne 1$.
    Does $\Q{\psi}{\varphi}{\alpha}{z} = \psi(1)^{\alpha}\varphi(1)^{1-\alpha}$
    or equivalently
    $\D{\psi}{\varphi}{\alpha}{z} = \log(\psi(1)/\varphi(1))$
    hold only if $(1/\psi(1))\psi = (1/\varphi(1))\varphi$ holds
    when $\psi$, $\varphi \ne 0$?
\end{Ques}
It is known that this is in the affirmative in the infinite dimensional type I case
(see \cite[Corollary 3.27]{Mosonyi23} and \cite[Theorem 2.1(1)]{ZhangQi23}).
\begin{Cmt}
    The equality condition of $\Q{\psi}{\varphi}{\alpha}{z} 
    = \psi(1)^{\alpha}\varphi(1)^{1-\alpha}$ is the
    same as the equality condition of generalized H\"{o}lder's inequality
    since we can rewrite it in terms of $L^p$-norms. 
    We give an equality condition in \ref{Appendix Holder} 
    when $r \ge 1$.
    However, we could not give it for all $p$, $q$, $r \in (0, \infty]$ with
    $1/r = 1/p + 1/q$.
\end{Cmt}
We could not show the DPI and the joint convexity 
for $\alpha > 1$.
\begin{Ques} \label{question DPI}
    Does our $\alpha$-$z$-R\'{e}nyi divergence also satisfy the DPI 
    under unital normal (completely) positive map $\gamma$ 
    if $\alpha >1$ and $\max\{\alpha-1, \alpha/2\} \le z \le \alpha$? 
\end{Ques}
The question was settled in the affirmative 
in the finite dimensional case by Zhang in \cite[Theorem 1.2]{Zhang20}. 
Moreover, it has been known, in the finite dimensional case, 
that the $\alpha$-$z$-R\'{e}nyi divergence satisfies the DPI 
under unital completely positive maps 
if and only if ($0 < \alpha <1$ and $z \ge \max\{\alpha, 1-\alpha\}$) 
or ($\alpha >1$ and $\max\{\alpha-1, \alpha/2\} \le z \le \alpha$). 
\begin{Cmt}
    If Question \ref{question variational expression} was settled in the affirmative,
    then one would only have to show the following:  
    \begin{equation} \label{question ineq}
            \|h_{\psi}^{1/2p} \gamma(b) h_{\psi}^{1/2p} \|_p
            \ge \| h_{\psi \circ \gamma}^{1/2p} 
            \, b \, h_{\psi\circ\gamma}^{1/2p}\|_p 
    \end{equation}
    for any $p \in [1/2, 1]$, $b \in s(\psi \circ \gamma) \N s(\psi \circ \gamma)$ 
    and any unital normal (completely) positive map $\gamma$. 
    It is no wonder that we consider this inequality. 
    \par
    Applying inequality \eqref{question ineq} to 
    the unital normal completely positive map 
    $\gamma \colon \M \longrightarrow \M \oplus \M$ defined by 
    $\gamma(a) := a \oplus a$ and 
    $\psi := \lambda \psi_1 \oplus (1-\lambda)\psi_2 \in (\M \oplus \M)_*^+$ gives 
    \begin{equation*}
        \lambda \tr((b^{1/2}h_{\psi_1}^{1/p} b^{1/2})^{p})
        + (1-\lambda) \tr((b^{1/2} h_{\psi_2}^{1/p} b^{1/2})^{p})
        \ge \tr((b^{1/2}(\lambda h_{\psi_1} 
        + (1-\lambda) h_{\psi_2})^{1/p} b^{1/2})^{p}).
    \end{equation*}
    It is known that this holds for any $1/2 \le p \le 1$ 
    in the finite dimensional case 
    due to Carlen and Lieb \cite[Theorem 1.1]{CarlenLieb08}, 
    and it was used to prove the DPI in the finite dimensional case 
    by Zhang \cite[pp.~12--14]{Zhang20}.
    (Remark that this inequality is closely related to the operator convexity 
    of $f(t) = t^{1/p}$, and thus does not hold when $p<1/2$.)
\end{Cmt}
\begin{Rmk}
    We first remark that the joint convexity can be obtained as  
    a corollary of the DPI. 
    In the finite dimensional case, it is known, 
    see \cite[Proposition 7]{CarlenFrankLieb18} or \cite[Proposition 2.1]{Zhang20}, 
    that the $\alpha$-$z$-R\'{e}nyi divergence satisfies the DPI 
    under unital completely positive maps
    if and only if 
    it has the joint convexity (resp.~the joint concavity), 
    when $\alpha > 1$ (resp.~$0 < \alpha < 1$).
\end{Rmk}
When $\alpha>1$, the $\alpha$-$z$-R\'{e}nyi divergence should have the monotonicity in $z$, which corresponds to Theorem \ref{Thm1}\ref{monotonicity1}.
\begin{Ques} \label{question monotonicity}
    When $\alpha > 1$, if $0 \le z \le z'$, then
    does $\Q{\psi}{\varphi}{\alpha}{z} \ge \Q{\psi}{\varphi}{\alpha}{z'}$ or 
    equivalently $\D{\psi}{\varphi}{\alpha}{z} \ge \D{\psi}{\varphi}{\alpha}{z'}$ 
    hold?
\end{Ques}
This is in the affirmative in the finite dimensional case 
(see \cite[Proposition 4.31(1)]{jaksicetal11}) 
and even in the infinite dimensional type I case 
(see \cite[Proposition 3.16]{Mosonyi23}). 
\begin{Cmt}
    The ALT inequality was used to prove the monotonicity in $z$ in 
    \cite{jaksicetal11} 
    and a certain approximation argument in addition was also 
    used in \cite{Mosonyi23}.
    However, as mentioned in \cite[Remark 3.18(1)]{Hiai21}, 
    we cannot use the ALT inequality in the von Neumann algebra setting.
    Besides, this question can be understood as a generalization of 
    the relation between the R\'{e}nyi divergence and
    the sandwiched R\'{e}nyi divergence 
    (see Remark \ref{rmk of monotonicity}).
    In \cite{BST18}, \cite{Jencova18}, a 
    complex analysis method (Hadamard's three line theorem) was used
    to prove this relation without the ALT inequality.
\end{Cmt}
Here is a question on the additivity under tensor products 
(Proposition \ref{tensor product alpha z}).
\begin{Ques}
    Does our $\alpha$-$z$-R\'{e}nyi divergence enjoy 
    the additivity under tensor products in the full generality? 
    Namely, do equalities 
    \eqref{multiplicativity under tensor product} 
    and \eqref{additivity under tensor product} hold 
    for any $\alpha$, $z > 0$ with $\alpha \ne 1$? 
\end{Ques}
In the finite dimensional case or more generally  
the infinite dimensional type I case, 
equalities
\eqref{multiplicativity under tensor product} 
and \eqref{additivity under tensor product} hold
for any $\alpha$, $z > 0$ with $\alpha \ne 1$ 
without any additional assumption. 
See \cite[axiom (V)]{AudenaertDatta15}, 
\cite[Lemma 3.22]{Mosonyi23}
and \cite[Theorem 2.1(5)]{ZhangQi23}.
\subsection*{Acknowledgements}
We would like to thank Professor Fumio Hiai for 
giving his online lectures during Feb.--Mar., 2022, 
which gave us the original motivation to this work
as well as for his comments to a draft version of this paper, 
and also Professor Yoshimichi Ueda for 
helpful discussions.
\setcounter{section}{0}
\renewcommand{\thesection}{Appendix \Alph{section}}
\section{A certain inequality via Petz's recovery map} \label{Appendix}
We will show a certain trace inequality (or $L^p$-norm inequality) used in 
the proof of Theorem \ref{Thm1}\ref{DPI1} 
via Petz's recovery map in this section. 
This appendix is based on Hiai's online lectures
entitled `Quantum Analysis and Quantum Information Theory' during
Feb.--Mar., 2022.
We also refer to \cite[Section 3.3]{Jencova18}, 
\cite[Section 4.2]{Jencova21}, 
\cite[Section 6.1, Lemma 8.3]{Hiai21} for the materials below. 
Thus, we do not claim any credit to the contents of this appendix.
\par
In this section, 
let $\M$, $\N$ be von Neumann algebras 
and $\gamma \colon \N \longrightarrow \M$ be a unital normal positive map
and $\gamma_* \colon L^1(\M) \simeq \M_* \longrightarrow L^1(\N) \simeq \N_* $ 
be the predual map of $\gamma$. 
Namely, $\gamma_*(h_{\varphi}) = h_{\varphi \circ \gamma}$ holds 
for every $\varphi \in \M_*$, and we immediately see that 
$\gamma_*$ preserves the $\tr$-functional.
\par
Next, we define Petz's recovery map (or the Petz dual) of $\gamma$.
Here, we employ the definition of Petz's recovery map in 
\cite[Lemma 8.3, equation (8.5)]{Hiai21}. 
\begin{Def}
    For every $\varphi \in \M_*^+$, 
    \emph{Petz's recovery map} of $\gamma$ with respect to $\varphi$ 
    is a unital normal positive map 
    $\gamma_{\varphi}^\star \colon s(\varphi) \M s(\varphi) \longrightarrow 
    s(\varphi \circ \gamma) \N s(\varphi \circ \gamma)$
    such that 
    \begin{equation} \label{Petz dual}
        h_{\varphi \circ \gamma}^{1/2} \gamma_{\varphi}^\star (a) 
        h_{\varphi \circ \gamma}^{1/2}
        = \gamma_* (h_{\varphi}^{1/2} a h_{\varphi}^{1/2})
        \qquad(a \in s(\varphi) \M s(\varphi)).
    \end{equation}
\end{Def}
\begin{Rmk}
    Petz's recovery map is uniquely defined 
    (see \cite[Proposition 6.6]{Hiai21}), 
    and hence this definition is well defined.
\end{Rmk}
The next lemma was given in \cite[Proposition 6.3]{Hiai21}
(cf.~\cite[Remark 6.7]{Hiai21}).
\begin{Lem}
    Let $\gamma \colon \N \longrightarrow \M$ 
    be a unital normal positive map,
    $\varphi \in \M_*^+$ and $\gamma_{\varphi}^\star$ be Petz's recovery map.
    Then, we have
    \begin{enumerate}[label = \rm{(\arabic*)}]
        \item $\varphi \circ \gamma \circ \gamma_{\varphi}^\star =
        \varphi \!\upharpoonright_{s(\varphi) \M s(\varphi)}$.
        \item $(\gamma_\varphi^\star)_{\varphi\circ\gamma}^\star = 
        s(\varphi) \gamma(\cdot) s(\varphi) \!\upharpoonright_{
        s(\varphi \circ \gamma) \N s(\varphi \circ \gamma)}$, 
        i.e., 
        $s(\varphi) \gamma(\cdot) s(\varphi) \!\upharpoonright_{s(\varphi \circ \gamma) \N s(\varphi \circ \gamma)} \colon 
        s(\varphi \circ \gamma) \N s(\varphi \circ \gamma)
        \longrightarrow s(\varphi) \M s(\varphi)$ 
        is Petz's recovery map of $\gamma_{\varphi}^\star$ with respect to
        $\varphi \circ \gamma \!\upharpoonright_{s(\varphi \circ \gamma) \N s(\varphi \circ \gamma)}$. 
    \end{enumerate}
\end{Lem}
By the above lemma and applying formula \eqref{Petz dual} to 
$\gamma_{\varphi}^\star$ and $\varphi \circ \gamma$, 
we can rewrite formula \eqref{Petz dual} as
\begin{equation} \label{Petz dual prime}
    h_{\varphi}^{1/2} \gamma(b) h_{\varphi}^{1/2} 
    = (\gamma_{\varphi}^{\star})_* (h_{\varphi \circ \gamma}^{1/2} 
    \, b \, h_{\varphi\circ\gamma}^{1/2})
    \qquad(b \in s(\varphi\circ\gamma) \N b \in s(\varphi\circ\gamma)).
\end{equation}
The next lemma is due to Jen\v{c}ov\'{a} 
\cite[Proposition 3.12]{Jencova18}.
\begin{Lem} \label{Jencova prop3.12}
    Let $p \in [1, \infty]$, 
    $L^p(\M, \varphi)_{1/2}$ be the symmetric Kosaki non-commutative
    $L^p$-space 
    and $\gamma \colon \N \longrightarrow \M$ 
    be a unital normal positive map.
    Then $\gamma_*$ maps $L^1(\M, \varphi)_{1/2} $ 
    into $L^1(\N, \varphi\circ \gamma)_{1/2}$
    and $\gamma_*$ restricts to a contraciton
    $L^p(\M, \varphi)_{1/2} \longrightarrow L^p(\N, \varphi\circ \gamma)_{1/2}$.
\end{Lem}
By formula \eqref{Petz dual prime} and applying 
Lemma \ref{Jencova prop3.12} to $\gamma_{\varphi}^\star$, 
we have
\begin{align*}
    \|h_{\varphi}^{1/2} \gamma(b) h_{\varphi}^{1/2} \|_{p, \varphi, 1/2}
    = \|(\gamma_{\varphi}^{\star})_* (h_{\varphi \circ \gamma}^{1/2} 
    \, b \, h_{\varphi\circ\gamma}^{1/2})\|_{p, \varphi, 1/2} 
    \le \| h_{\varphi \circ \gamma}^{1/2} 
    \, b \, h_{\varphi\circ\gamma}^{1/2}\|_{p, \varphi\circ\gamma, 1/2}
\end{align*}
for every $b \in s(\varphi \circ \gamma) \N s(\varphi \circ \gamma)$.
By the identification of Kosaki and Haagerup 
non-commutative $L^p$-spaces, 
we have, for every $p \in [1, \infty]$, 
\begin{equation} \label{appendix inequality}
    \|h_{\varphi}^{1/2p} \gamma(b) h_{\varphi}^{1/2p} \|_p
    \le \| h_{\varphi \circ \gamma}^{1/2p} 
    \, b \, h_{\varphi\circ\gamma}^{1/2p}\|_p  
    \qquad( b \in s(\varphi \circ \gamma) \N s(\varphi \circ \gamma) ).
\end{equation}
\begin{Rmk}
    In the finite dimensional case,
    a similar inequality is given in \cite[Theorem 1.9]{CarlenZhang22}.
    Let $p\ge 1$, $\Phi \colon M_m(\mathbb{C}) \longrightarrow M_n(\mathbb{C})$
    be a unital Schwarz map 
    (i.e., $\Phi(B^*B) \ge \Phi(B)^*\Phi(B)$ for each $B \in M_m(\mathbb{C})$)
    and $\Phi^*$ be the adjoint map of $\Phi$.
    Then, for any positive invertible $A \in M_n(\mathbb{C})$  and
    any $B \in M_m(\mathbb{C})$, we have
    \begin{equation} \label{CarlenZhang}
        \operatorname{Tr}((\Phi(B)^*A^{1/p}\Phi(B))^{p})
        \le \operatorname{Tr}((B^*\Phi^*(A)^{1/p}B)^{p}).
    \end{equation}
    By Lemma \ref{trace equality} (slightly extended) and the Schwarz property, 
    \begin{equation*}
        \text{LHS of \eqref{CarlenZhang}}
        = \operatorname{Tr}((A^{1/2p}\Phi(B)\Phi(B)^*A^{1/2p})^{p})
        \le \operatorname{Tr}((A^{1/2p}\Phi(BB^*)A^{1/2p})^{p}).
    \end{equation*}
    Let $E:= s(\Phi^*(A)) = \Phi^*(A)^0 \in M_m(\mathbb{C})$ 
    be the support projection of $\Phi^*(A)$.
    If $BB^* \in E M_m(\mathbb{C}) E$, then 
    \begin{equation*}
        \operatorname{Tr}((A^{1/2p}\Phi(BB^*)A^{1/2p})^{p})
        \le \operatorname{Tr}((\Phi^*(A)^{1/2p}BB^*\Phi^*(A)^{1/2p})^{p})
        = \text{RHS of \eqref{CarlenZhang}}
    \end{equation*}
    by \eqref{appendix inequality}. 
    Namely, inequality \eqref{CarlenZhang} is related to 
    inequality \eqref{appendix inequality}.
\end{Rmk}
\section{Equality in H\"{o}lder's inequality} \label{Appendix Holder}
In this appendix, we will give an equality condition of H\"{o}lder's inequality
in the framework of the Haagerup non-comuutative $L^p$-spaces.
The case of non-commutative $L^p$-spaces 
asociated with semifinite von Neumann algebras
was discussed by Larotonda \cite{Larotonda16} 
and historical comments therein. 
We are grateful to Professor Yoshimichi Ueda 
for several advices about this section.
\begin{Thm} \label{equality Holder thm}
    Let $p$, $q \in (1, \infty)$, $r \ge 1$ with $1/r = 1/p + 1/q$,
    $x \in L^p(\M)_+$ and $y \in L^q(\M)_+$.
    If $\|xy\|_r = \|x\|_p \|y\|_q$, 
    then $x^p = \lambda y^q$ or $y^q = \lambda x^p$ for some $\lambda \ge 0$.
\end{Thm}
\begin{proof}
    If $y = 0$, then $y = \lambda x$ holds with $\lambda = 0$.
    If $x = 0$, then $x = \mu y$ holds with $\mu = 0$.
    Thus, we may and do assume that $x \ne 0$, $y \ne 0$. 
    \par
    We firstly consider the case of $r = 1$.
    We may and do also assume that $x = h_{\varphi}^{1/p}$, $y = h_{\psi}^{1/q}$
    with $\varphi$, $\psi \in \M_*^+$, $\varphi \ne 0$ and $\psi \ne 0$. 
    Let $xy = u |xy|$ be the polar decompostion.
    Consider $f(z) = \tr (u^* h_{\varphi}^z h_{\psi}^{1-z})$,
    a bounded continuous function over $0 \le \Re z \le 1$, 
    analytic in the interior. 
    We have
    \begin{equation*}
        |f(it)| \le \psi(1), \qquad |f(1 + it)| \le \varphi(1).
    \end{equation*}
    The famous Hadamard three line theorem shows 
    $|f(z)| \le \varphi(1)^{\Re z} \psi(1)^{1 -\Re z}$. 
    \par
    Then, $g(z) = \varphi(1)^{-z} \psi(1)^{-(1 -z)} f(z)$ satisfies
    the same properties $f(z)$ does and
    \begin{equation*}
        |g(z)| \le 1
        \qquad(0 \le \Re z \le 1).
    \end{equation*}
    The assumption here implies
    \begin{equation*}
        g(1/p) = \varphi(1)^{-1/p} \psi(1)^{-1/q} f(1/p)
        = \|x\|_p^{-1} \|y\|_q^{-1} \|xy\|_1 = 1.
    \end{equation*}
    By the boundary value property in complex analysis,
    \begin{equation*}
        1 = |g(1/2)| = \varphi(1)^{-1/2} \psi(1)^{-1/2} 
        |\tr (u^* h_{\varphi}^{1/2} h_{\psi}^{1/2})|.
    \end{equation*}
    Hence
    \begin{equation*}
        \| h_{\varphi}^{1/2}\|_2 \| h_{\psi}^{1/2}\|_2 
        = |\tr (u^* h_{\varphi}^{1/2} h_{\psi}^{1/2})| 
        \le \|u^* h_{\varphi}^{1/2} \|_2 \| h_{\psi}^{1/2} \|_2
        \le \| h_{\varphi}^{1/2} \|_2 \| h_{\psi}^{1/2} \|_2
    \end{equation*}
    and the equality condition of the Cauchy--Schwarz inequality implies
    $u^* h_{\varphi}^{1/2} = \alpha h_{\psi}^{1/2}$ for some $\alpha \in \mathbb{C}^{\times}$.
    Since $|\tr (u^* h_{\varphi}^{1/2} h_{\psi}^{1/2})| 
    = |\tr (h_{\varphi}^{1/2} h_{\psi}^{1/2} u^*)|$,
    we also have $h_{\psi}^{1/2}u^* = \beta h_{\varphi}^{1/2}$ 
    for some $\beta \in \mathbb{C}^{\times}$. 
    \par
    By the construction of the polar decomposition,
    $uu^* \le s(h_{\varphi}^{1/2}) = s(h_{\varphi}) = s(\varphi)$.
    Since
    \begin{equation*}
        h_{|\beta|^2 \varphi} = |\beta|^2 h_{\varphi}
        = uh_{\psi} u^* = h_{u\psi u^*},
    \end{equation*}
    we observe that
    \begin{equation*}
        |\beta|^2\varphi(s(\varphi) -uu^*) = u\psi u^* (s(\varphi) -uu^*) 
        = \psi(u^* s(\varphi)u - u^* u) 
        = \psi (u^*u - u^*u) = 0,
    \end{equation*}
    implying that $s(\varphi) = u u^*$.
    Therefore, 
    \begin{equation*}
        |\alpha|^2 h_{\psi} = (\alpha h_{\psi}^{1/2})^* (\alpha h_{\psi}^{1/2}) 
        = (u^* h_{\varphi}^{1/2})^* (u^*h_{\varphi}^{1/2}) 
        = h_{\varphi}^{1/2} u u^* h_{\varphi}^{1/2} = h_{\varphi},
    \end{equation*}
    and thus 
    \begin{equation*}
        |\alpha|^2 y^q = |\alpha|^2 (h_{\psi}^{1/q})^q
        = |\alpha|^2 h_{\psi} = h_{\varphi}
        = (h_{\varphi}^{1/p})^p = x^p.
    \end{equation*}
    Hence we are done when $r = 1$.
    \par
    We next consider the case of $r > 1$. 
    By the Araki--Lieb--Thirring inequality (see \cite[Theorem 4]{Kosaki92})
    and H\"{o}lder's inequality, we have
    \begin{equation*}
        \|x\|_p \|y\|_q = \|xy\|_r = \| |xy|^r \|_1^{1/r} 
        \le \|x^r y^r\|_1^{1/r}
        \le \|x^r\|_{p/r}^{1/r} \|y^r\|_{q/r}^{1/r} 
        = \|x\|_p \|y\|_q
    \end{equation*}
    and thus
    $\|x^r y^r\|_1 = \|x^r\|_{p/r} \|y^r\|_{q/r}$.
    Hence we also have the desired assertion by the first case.
\end{proof}
Let $a \in L^p(\M)$, $b \in L^q(\M)$ with $1/p +1/q =1$.
Let $b = v|b|$ be the polar decomposition. 
We observe that 
$|ab^*|^2 = v |b| |a|^2 |b| v^* = v ||a||b||^2 v^*$.
Hence, $\| |ab^*|\|_r = \| |a||b|\|_r$.
Therefore, equality in H\"{o}lder's inequality 
$\|ab^* \|_r \le \|a \|_p \|b\|_q$ implies 
$|a|^p = \lambda |b|^q$ or $|b|^q = \lambda |a|^p$ for some $\lambda \ge 0$
when $r \ge 1$.
Consequently, we get:
\begin{Cor}[equality condition for H\"{o}lder's inequality] \label{equality condition for Holder}
    Let $p$, $q \in (1, \infty)$, $r \ge 1$ with $1/r = 1/p + 1/q$,
    $a \in L^p(\M)$ and $b \in L^q(\M)$. 
    The following conditions are equivalent:
    \begin{enumerate}[label=\rm{(\arabic*)}]
        \item $\|ab^*\|_r = \|a\|_p \|b\|_q$.
        \item $|a|^p = \lambda |b|^q$ or $|b|^q = \lambda |a|^p$ for some $\lambda \ge 0$.
    \end{enumerate}
\end{Cor}
\printbibliography
\end{document}